\documentclass[12pt]{article}
\usepackage[margin=1.5cm]{caption}
\usepackage[utf8]{inputenc}

\usepackage{amssymb}
\usepackage{mystyle}

\usepackage[a4paper, textwidth=16cm]{geometry}
\usepackage{tikz}
\usetikzlibrary{calc, arrows.meta, positioning, cd, external, backgrounds}
\tikzset{
	commutative diagrams/.cd,
	arrow style=tikz,
	diagrams={>={Stealth}},
}
\tikzset{external/disable dependency files}
\tikzset{external/aux in dpth={false}}

\newacro{B-spline}{Basis spline}
\newacro{CAD}{Computer Aided Design}
\newacro{DoF}{Degree of Freedom}
\acrodefplural{DoF}[DoFs]{Degrees of Freedom}
\newacro{FEM}{Finite Element Method}
\newacro{HE}{Hard Edge}
\newacro{IGA}{IsoGeometric Analysis}
\newacro{NURBS}{Non-Uniform Rational Basis Spline}
\newacro{ODE}{Ordinary Differential Equation}
\newacro{PEC}{Perfect Electric Conductor}
\newacro{SI}{Syst\'eme international (d'unit\'e)}
\begin{document}
\title{IsoGeometric Approximations for Electromagnetic Problems in Axisymmetric Domains}

\author{Abele Simona $^{(1),(2)}$, Luca Bonaventura$^ {(1)}$\\
		Carlo de Falco  $^{(1)}$, Sebastian Sch\"ops $^{(2)}$}

\maketitle

\begin{center}
	{\small
		
		(1) MOX -- Modelling and Scientific Computing, \\
		Dipartimento di Matematica, Politecnico di Milano \\
		Via Bonardi 9, 20133 Milano, Italy\\
		{\tt abele.simona@polimi.it, luca.bonaventura@polimi.it, carlo.defalco@gmail.com }\\
		{$ \ \ $ }\\
		(2)  Technische Universit\"at Darmstadt \\
		Centre for Computational Engineering \\         
		Dolivostra{\ss}e 15,   64293 Darmstadt, Germany  \\           
		{\tt  abele@gsc.tu-darmstadt.de, schoeps@gsc.tu-darmstadt.de }\\
	}
\end{center}

\date{}

\vspace*{1.5cm}

\vspace*{0.5cm}


\pagebreak

\abstract{
\acresetall  
We propose a  numerical method for the solution of electromagnetic problems
on axisymmetric domains, based on a
combination of  a spectral Fourier
approximation in the azimuthal direction with an 
\ac{IGA} approach in the radial and axial
directions.
This  combination allows to blend the flexibility and accuracy of \ac{IGA} approaches with the advantages of a Fourier
representation on axisymmetric domains. It also allows to reduce significantly the computational cost 
by decoupling of the computations required for each Fourier mode.  We prove that the discrete approximation
spaces employed  functional space constitute a closed and exact de Rham sequence. Numerical simulations of relevant
benchmarks confirm the high order convergence  and other computational advantages of the  proposed method.
}

\pagebreak

\section{Introduction}
\label{intro}
\acresetall  

The solution of the Maxwell equations on axisymmetric domains, as the one depicted in Figure~\ref{fig:axis_scheme}, plays an
important role in many applications, such as, for example, electrical machines, 
cables, the beam pipe of a particle accelerator magnet or resonating cavities.
\begin{figure}[!htb]
	\centering
	\includegraphics[width=0.75\textwidth]{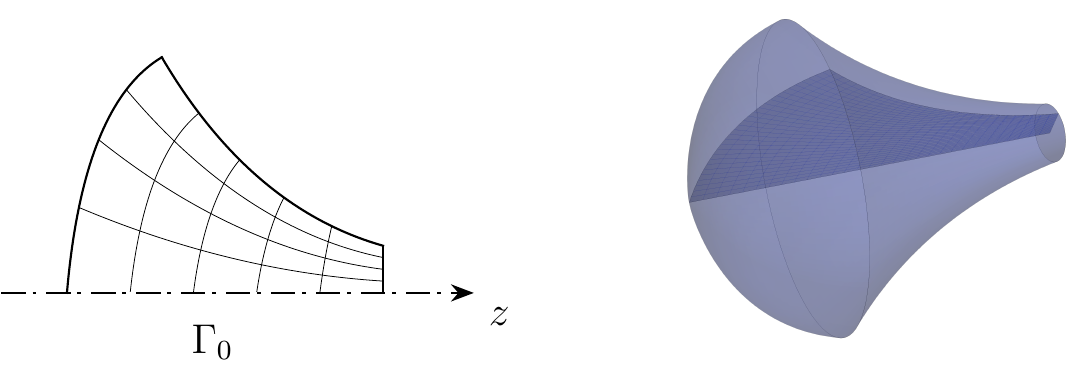}
	\caption{Representation of a cross-section $\srfdomflat$ lying on the $\parGs{\rho=0}$ axis (right) and its associated axisymmetric domain $\voldomaxis$ (right). $\Gamma_0$ represents the portion of the boundary of $\srfdomflat$ on the axis.}
	\label{fig:axis_scheme}
\end{figure}%
Several tailored numerical techniques to solve electromagnetic problems on such domains have been
proposed in the literature, \textit{e.g.}, among many others,
\cite{assous:2002,bernardi:1999,borm:2002,copeland:2006,copeland:2010,gopalakrishnan:2012,li:2011,mercier:1982,oh:2015}. In
particular, when considering the computation of multipole expansions 
in particle accelerator magnets, often one has to consider boundary
conditions that are accurately represented by a combination of a small
number of Fourier modes. Therefore, if a Fourier spectral approach is
employed in the azimuthal direction, the resulting computational cost can be
significantly reduced.  

Concerning specifically particle accelerator simulation, it has been recently shown
in \cite{abele:2019} that the accuracy of particle tracking methods
employed to design new accelerator configurations is essentially
limited by the accuracy of the field reconstruction. In particular,
numerical approximations of electric and magnetic fields that do not respect the inherent
mathematical structure of the Maxwell equations may disrupt the effective accuracy of the
simulation. This motivates the search for consistent numerical methods that
provide high order accuracy in the spatial approximation.

In this work, we propose a  combination of  a spectral Fourier
approximation in the azimuthal direction with an 
\ac{IGA} approach in the radial and axial
directions, extending the strategy presented in~\cite{oh:2015}. The combination of the Fourier basis and \ac{IGA} is in the spirit of the recent developments 
in \cite{perotto:2017} in computational fluid dynamics.
\ac{IGA} has become popular in applied
electromagnetism in the last decade 
\cite{bontinck:2017}, \cite{buffa:2010}. One of its features is that it  provides
discrete approximations that, 
thanks to an appropriate choice of the finite-dimensional
approximation spaces, indeed satisfy discrete analogous 
of the continuous Maxwell equations. In more technical terms,
in the spirit of de Rham \cite{monk:2003},
sequences of functional spaces which are closed and exact %
can be approximated by closed and exact
sequences at a discrete level. 

The combination of splines and Fourier bases in cylindrical coordinates effectively mitigates two problems. On the one hand, electrotechnical devices, particularly those based on high-frequency electromagnetic fields, are very sensitive to shape variations and are usually described by solutions of high regularity \cite{corno:2016}, which favours splines-based field and geometry representations. On the other hand, \ac{CAD} tools rarely provide volumetric representations.  Therefore, considering the cylindrical coordinate system also helps in bridging the gap between \ac{CAD} and simulation.  
Furthermore, the orthogonality of the Fourier representation allows to decouple the computation associated to different Fourier modes.
Notice that, recently, the scaled boundary approach was proposed and demonstrated for Laplace-type problems in \cite{arioli:2019}. While not limited to axisymmetric domains, this approach does  not exploit the orthogonality of the Fourier basis and does not display high convergence rates of the spectral method proposed here.

The paper is structured as follows. In Section~\ref{sec:preliminaries}, we briefly recall the Maxwell equations and we introduce the notation for the treatment of problems in axisymmetric domains. Section~\ref{sec:discretization} is devoted to the presentation of the strategy used to extend a general discretization in Cartesian coordinate to the cylindrical setting. Moreover, appropriate error estimates are introduced. In Section~\ref{sec:iga}, we briefly introduce the discrete spaces used for \ac{IGA}, together with error estimates, that will be used to build the discretization presented in Section~\ref{sec:discretization}. In Section~\ref{sec:tests}, we apply the method presented in this work to relevant problems often arising in electromagnetism. Finally, in Section~\ref{sec:conclu}, we draw conclusions and discuss possible future developments of the proposed approach.

\section{Preliminaries}
\label{sec:preliminaries}
Consider the time-harmonic Maxwell equations~\cite{jackson:2007, monk:2003}
\begin{equation}
\label{eq:max_eq}
\begin{aligned}
\curlt \Ev & = i\omega\,{\Bv}~,\\
\divergt \Dv & =\chargeden~, \\
\curlt \Hv & = \jv -i\omega\,{\Dv}~,\\
\divergt \Bv & = 0~, 
\end{aligned}
\end{equation}
together with the magnetic vector potential $\magvp$, for which 
\begin{equation*}
\Bv = \curlt \magvp
\end{equation*}
holds, and a set of linear constitutive relations that links the fluxes $\Dv$ and $\Bv$ to the field strengths $\Ev$ and $\Hv$:
\begin{equation*}
\Dv = \pmitt \Ev~, \qquad \Bv = \pmeab \Hv~,
\end{equation*}
where $\pmitt$ and $\pmeab$ are the permittivity and permeability, respectively.
A first problem, which often arises in electromagnetics and that we will consider in Sections~\ref{sec:pillbox} and~\ref{sec:tesla_cavity}, amounts to determining the eigenmodes of a resonant cavity solving the following source problem:
\begin{equation}
\label{eq:curlcurl_E}
\curlt \parT{\pmeab^{-1}\curlt\Ev} = \omega^2\pmitt\Ev~,
\end{equation}
with suitable boundary conditions.
The second type of problem we will consider arises in the stationary limit $\omega \rightarrow 0$, and amounts to determine the magnetic vector potential $\magvp$ regularized by the Coulomb gauge, solving the following equations:
\begin{equation}
\label{eq:A_form_st}
\begin{aligned}
&\curlt \parT{\pmeab^{-1} \curlt \magvp} = \jv~, \\
&\divergt\parT{\pmitt\magvp} = 0~,
\end{aligned}
\end{equation}
for a given source current density $\jv$ and with suitable boundary conditions.

The problems above will be solved on axisymmetric  bounded Lipschitz domains $\voldomaxis \subset
\R^3$ represented in cylindrical coordinates. We will exploit the
Fourier basis to define different de Rham complexes for each mode $m$,
which will then be analysed and applied to some numerical examples in
the following sections. In this work we will restrict our attention to the case $m\neq0$, which corresponds to functions that are not axisymmetric.
Let us introduce the cylindrical coordinates $(\rho, \, z, \,
\theta)$, with the $z$-axis being the symmetry axis of the domain. The
uncommon choice of placing the angular variable $\theta$ as the last
one will considerably simplify the notation of the following considerations. 
The Cartesian coordinates are related to the cylindrical ones by
\begin{equation*}
\begin{bmatrix} x \\ y \\ z\end{bmatrix} =
\vec{g} \parT{\begin{bmatrix} \rho \\ z \\ \theta \end{bmatrix}}
= \begin{bmatrix} \rho \,\cos{\theta} \\ \rho \, \sin{\theta} \\
z\end{bmatrix}
\end{equation*}
and, conversely
\begin{equation*}
\rho = \sqrt{x^2 + y^2} \quad \text{and} \quad \theta =
\left\{
\begin{aligned}
&-\arccos \parT{\tfrac{x}{\rho}}~, && \text{if $y < 0$}~, \\
&\arccos \parT{\tfrac{x}{\rho}}~, && \text{if $y \geq 0$}~.
\end{aligned}
\right.
\end{equation*}
We describe the domain $\voldomaxis$ using its cross-section with
respect to the and $\rho z$-plane $\srfdomflat \subset \R^+ \times
\R$. 
We assume that $\srfdomflat \subset \R^2$ is a bounded Lipschitz domain obtained through a diffeomorphism $\Fv$ of the unit square and that $\partial \srfdomflat \cap \parG{\rho = 0}$ is either empty or coincides with $\Fv(\parGs{0} \times \parQ{0, \, 1})$.
Let $\Gamma_0 = \interior\parT{\partial \srfdomflat \cap \parG{\rho =
		0}}$ be the interior of the intersection of the boundary of $\srfdomflat$
with the $z$-axis and $\Gamma = \partial \srfdomflat \setminus \Gamma_0$. 
The volume $\voldomaxis$ is obtained by rotating the cross-section
over the symmetry axis $z$ and adding $\Gamma_0$: 
\begin{equation*}
\voldomaxis = \parG{ \xv \in \R^3: \xv = \vec{g}\parT{\parT{\rho, \,
			z, \, \theta}^T}, ~ \parT{\rho, \, z}^T \in \srfdomflat, ~\theta \in [0,
	\, 2\pi)} \cup \Gamma_0~. 
\end{equation*}
We have that $\partial \voldomaxis = \Gamma \times [0, \, 2\pi)$ (see Figure~\ref{fig:axis_scheme}).

\medskip
The differential operators involved in the standard de Rham complex in
Cartesian coordinates~\cite{buffa:2011}
\begin{equation}
\label{eq:3d_cart_seq}
\begin{split}
\tikzsetnextfilename{de_rham/de_rham_3d_cart}
\begin{tikzpicture}[commutative diagrams/every diagram,
arr/.style={commutative diagrams/.cd,
	every arrow,
	every label}]
\def\minW{25mm}
\def\minH{20mm}
\def\hShift{10mm}
\def\vShiftA{20mm}
\def\vShiftB{4mm}
\def\auxA{1mm}
\def\auxB{4mm}
%
\node (X0) [] at (0, 0) {$\Hsp^1\parTs{\voldomaxis}$};
\node (X1) [xshift=1.2*\minW] at (X0) {$\Hsp\parTs{\curlt;\,\voldomaxis}$};
\node (X2) [xshift=1.3*\minW] at (X1) {$\Hsp\parTs{\divergt;\,\voldomaxis}$};
\node (X3) [xshift=1*\minW] at (X2) {$\Lsp^2\parTs{\voldomaxis}$};
\node (R) [xshift=-0.6*\minW] at (X0) {$\R$};
\node (Z) [xshift=0.6*\minW] at (X3) {$0$};
\path[draw, arr] (X0) -- (X1) node [midway] {$\gradt$};
\path[draw, arr] (X1) -- (X2) node [midway] {$\curlt$};
\path[draw, arr] (X2) -- (X3) node [midway] {$\divergt$};
\path[draw, arr] (R) -- (X0);
\path[draw, arr] (X3) -- (Z);
\end{tikzpicture}
\end{split}
\end{equation}
correspond to the
following ones in cylindrical coordinates: 
\begin{equation*}
\label{eq:cyl_grad}
\renewcommand\arraystretch{1.5}
\gradc u = \begin{bmatrix}
\pd{\rho}{u} \\
\pd{z}{u} \\
\dfrac{1}{\rho}\pd{\theta}{u}
\end{bmatrix}~,
\end{equation*}
\begin{equation*}
\label{eq:cyl_curl}
\renewcommand\arraystretch{1.5}
\curlc \uv = \begin{bmatrix}
\dfrac{1}{\rho}\pd{\theta}{u_z} - \pd{z}{u_\theta}\\
\dfrac{1}{\rho}\parT{\pd{\rho}\parT{\rho \,u_\theta} -
	\pd{\theta}{u_\rho}}\\ 
\pd{z}{u_\rho} - \pd{\rho}{u_z}
\end{bmatrix}~,
\end{equation*}
\begin{equation*}
\label{eq:cyl_div}
\divergc \uv = \dfrac{1}{\rho}\, \pd{\rho}\parT{\rho \,
	u_\rho} + \dfrac{1}{\rho}\,\pd{\theta}{u_\theta} +
\pd{z}{u_z}~. 
\end{equation*}
Since the domain is axisymmetric, we can exploit the Fourier
orthogonal system in the angular variable $\theta$. Moreover, it is convenient to express each
function on $\voldomaxis$ as the sum of a symmetric and an
antisymmetric part with respect to the plane at $\theta=0$, so that, 
for a scalar function, we have that $u = u^s + u^a$, where
\begin{align*}
u^s & = u^\parTs{0} + \sum_{m=1}^{\infty} u^\m \,\cos{m\theta}~, \\
u^a & = \sum_{m=1}^{\infty}u^\mm \,\sin{m\theta}~.
\end{align*}
For a vector function we have instead that $\uv = \uv^s +
\uv^a$, where
\begin{align*}
\uv^s &= \begin{bmatrix}
u_\rho^\parTs{0} \\
u_z^\parTs{0}\\
0
\end{bmatrix}+\sum_{m=1}^{\infty}\begin{bmatrix}
u_\rho^\m \,\cos{m\theta}\\
u_z^\m\,\cos{m\theta}\\
u_\theta^\m\,\sin{m\theta}
\end{bmatrix}~, \\
\uv^a &= \begin{bmatrix}
0 \\
0\\
u_\theta^\parTs{0}
\end{bmatrix}+\sum_{m=1}^{\infty}\begin{bmatrix}
u_\rho^\mm \,\sin{m\theta}\\
u_z^\mm\,\sin{m\theta}\\
u_\theta^\mm\,\cos{m\theta}
\end{bmatrix}~.
\end{align*}
The Fourier coefficients $u^\m$, $\uv^\m$ are defined on the cross-section $\srfdomflat$. With this representation, the effect of the
differential operators on each mode $m$ can be considered
independently, leading to the definition of the following operators
acting on the Fourier coefficients: 
\begin{equation}
\label{eq:cyl_gradm}
\renewcommand\arraystretch{1.5}
\gradm u^\m = \begin{bmatrix}
\pd{\rho}{u^\m} \\
\pd{z}{u^\m}\\
-\dfrac{m}{\rho}u^\m
\end{bmatrix}~,
\end{equation}
\begin{equation}
\label{eq:cyl_curlm}
\renewcommand\arraystretch{1.5}
\curlm \uv^\m = \begin{bmatrix}
-\dfrac{m}{\rho}u^\m_z - \pd{z}{u^\m_\theta}\\
\dfrac{1}{\rho}\parT{\pd{\rho}\parT{\rho\,u^\m_\theta} +
	{m}\,{u^\m_\rho}}\\ 
\pd{z}{u^\m_\rho} - \pd{\rho}{u^\m_z}
\end{bmatrix}~,
\end{equation}
\begin{equation}
\label{eq:cyl_divm}
\divergm \uv^\m = \dfrac{1}{\rho}\,\pd{\rho}\parT{\rho\, u^\m_\rho} -
\dfrac{m}{\rho}{u^\m_\theta} + \pd{z}{u_z}~. 
\end{equation}
In cylindrical coordinates, the integral  of a function $f$ over a
volume has the form 
\begin{equation*}
\int_{\voldomaxis} f\, \rho \, \dd\rho \dd z \dd \theta~.
\end{equation*}
This leads naturally to the use of weighted Hilbert spaces, which are
defined as follows: 
\begin{equation*}
\begin{alignedat}{2}
&\Lsp_\rho^2(\srfdomflat) &&= \parG{u: \int_\srfdomflat
	u^2 \,\rho\,\dd\rho \dd z < \infty}~, \\ 
&\Lsp_\rho^2(\srfdomflat;\, \R^3) &&= \parG{\uv:
	\int_\srfdomflat\norm{\uv}_e^2\, \rho\,\dd\rho \dd z <
	\infty}~, 
\end{alignedat}
\end{equation*}
where $\norm{\uv}_e = \sqrt{\uv \cdot \uv}$ is the Euclidean norm and
$\uv$ is a vectorial function defined on $\srfdomflat$ which takes
values in $\R^3$. The associated norms are the standard ones induced
by the inner products
\begin{equation*}
\idot{u}{v}_{\Lsp_\rho^2(\srfdomflat)} = \int_\srfdomflat uv
\,\rho\,\dd\rho \dd z~, \quad
\idot{\uv}{\vv}_{\Lsp_\rho^2(\srfdomflat;\, \R^3)} =
\int_\srfdomflat \uv \cdot \vv \,\rho\,\dd\rho \dd z~, 
\end{equation*}
and will be often simply denoted using the subscript $\rho$, \textit{i.e.}
$\norm{u}_\rho = \sqrt{\idot{u}{u}_\rho}$. 
We also introduce the following spaces:
\begin{equation}
\label{eq:Zm_notation}
\begin{alignedat}{3}
&\CySpT^{m, \, 0} && = \Hsp_\rho(\gradm;\,\srfdomflat)&& =
\parG{u \in \Lsp_\rho^2(\srfdomflat): \gradm u \in
	\Lsp_\rho^2(\srfdomflat; \,\R^3)},\\ 
&\CySpT^{m, \, 1} && = \Hsp_\rho(\curlm;\,\srfdomflat) &&=
\parG{\uv \in \Lsp_\rho^2(\srfdomflat; \,\R^3): \curlm \uv \in
	\Lsp_\rho^2(\srfdomflat; \,\R^3)},\\ 
&\CySpT^{m, \, 2} && = \Hsp_\rho(\divergm;\,\srfdomflat)&& =
\parG{\uv \in \Lsp_\rho^2(\srfdomflat; \,\R^3): \divergm \uv \in
	\Lsp_\rho^2(S)}~,\\ 
&\CySpT^{m, \, 3} && = \Lsp_\rho^2(\srfdomflat)~,&&
\end{alignedat}
\end{equation}
where the differential operators have to be interpreted in a weak
sense and the corresponding norms are 
\begin{alignat*}{2}
&\norm{u}_{\Hsp_\rho(\gradm;\,\srfdomflat)}^2 &&= \norm{u}_\rho^2 +
\norm{\gradm u}_\rho^2~,\\ 
&\norm{\uv}_{\Hsp_\rho(\curlm;\,\srfdomflat)}^2 &&= \norm{\uv}_\rho^2
+ \norm{\curlm u}_\rho^2~,\\ 
&\norm{\uv}_{\Hsp_\rho(\divergm;\,\srfdomflat)}^2 &&=
\norm{\uv}_\rho^2 + \norm{\divergm u}_\rho^2~. 
\end{alignat*}
The spaces~\eqref{eq:Zm_notation} arise naturally in the resolution of
electromagnetic problems in cylindrical coordinates by means of
Galerkin formulations.
For example, for $m\neq0$, the formulation associated to~\eqref{eq:curlcurl_E} with \ac{PEC} boundary conditions reads
\begin{problem}
	\label{pb:weak_curl_curl}
	Find the eigenpair $\parTs{\omega^2, \, \Ev^{(m)}} \in \R_+ \times \CySpT_0^{m, \, 1}$,
	with $\Ev \neq \vec{0}$ and where the subscript $0$ in the space denotes the vanishing tangential component on $\Gamma$, such that
	\begin{equation}
	\label{eq:axis_weak_eig_pb2}
	\begin{split}
	&\int_\srfdomflat \parTs{\pmeab^{-1}\curltm \Ev^{(m)}}\cdot\curltm \vv\,\rho\,\dd\rho \dd z \\
	&\hspace{30mm}= \omega^2 \int_\srfdomflat \pmitt \Ev^{(m)} \cdot \vv\,\rho\,\dd\rho \dd z
	\end{split}~,\quad \vv \in \CySpT_0^{m, \, 1}~.
	\end{equation}
\end{problem}
The Maxwell equations~\eqref{eq:max_eq}, 
through the differential operators~\eqref{eq:cyl_gradm}--\eqref{eq:cyl_divm},
relate quantities belonging to different spaces. A representation of
the resulting structure is given by the sequence
\begin{equation*}
\label{eq:3d_cyl_seq}
\begin{split}
\tikzsetnextfilename{de_rham/de_rham_3d_cyl}
\begin{tikzpicture}[commutative diagrams/every diagram,
arr/.style={commutative diagrams/.cd,
	every arrow,
	every label}]
\def\minW{25mm}
\def\minH{20mm}
\def\hShift{10mm}
\def\vShiftA{20mm}
\def\vShiftB{4mm}
\def\auxA{1mm}
\def\auxB{4mm}
%
\node (X0) [] at (0, 0) {$\CySpT^{m, \, 0}$};
\node (X1) [xshift=1.1*\minW] at (X0) {$\CySpT^{m, \, 1}$};
\node (X2) [xshift=1.1*\minW] at (X1) {$\CySpT^{m, \, 2}$};
\node (X3) [xshift=1.1*\minW] at (X2) {$\CySpT^{m, \, 3}$};
\node (R) [xshift=-0.6*\minW] at (X0) {$0$};
\node (Z) [xshift=0.6*\minW] at (X3) {$0$};
\path[draw, arr] (X0) -- (X1) node [midway] {$\gradtm$};
\path[draw, arr] (X1) -- (X2) node [midway] {$\curltm$};
\path[draw, arr] (X2) -- (X3) node [midway] {$\divergtm$};
\path[draw, arr] (R) -- (X0);
\path[draw, arr] (X3) -- (Z);
\end{tikzpicture}
\end{split}
\end{equation*}
In \cite[Theorem 2.1]{oh:2015}, it has been proven that the 
sequence is exact for $m\neq0$, meaning that the range of each operator linking two
spaces is equal to the kernel of the subsequent one. 
The use of discretizations that respect the above structure is necessary to yield
stable numerical methods. The aim of the following sections is to
present a technique to build such discretizations, which are well-suited for the resolution of problems in electromagnetism.

\section{Discretization}
\label{sec:discretization}
To  solve numerical problems in electromagnetism, we need to
approximate the continuous spaces $\Zsp^{m, \, k}$, $k=0,\, \ldots, \,
3$ by means of finite-dimensional subspaces $\Zsp_h^{m, \, k}$ which
depend on a discretization parameter related to the characteristic
element size $h$. In the following, we will use the subscript $h$ to
refer to quantities related to the finite-dimensional spaces and we
will consider conforming discretizations, meaning that $\Zsp_h^{m, \,
	k} \subset \Zsp^{m, \, k}$ for every value of $h$. 
In this section, we mimic the strategy used in \cite{oh:2015} to build
an exact sequence of conforming discrete spaces and projectors such
that, for sufficiently regular functions, the following diagram commutes
\begin{equation}
\label{eq:cyl_sequence}
\begin{split}
\tikzsetnextfilename{de_rham/de_rham_3d_cyl_cont_discr}
\begin{tikzpicture}[commutative diagrams/every diagram,
arr/.style={commutative diagrams/.cd,
	every arrow,
	every label}]
\def\minW{25mm}
\def\minH{15mm}
\def\hShift{10mm}
\def\vShiftA{20mm}
\def\vShiftB{4mm}
\def\auxA{1mm}
\def\auxB{4mm}
%
\node (X0) [] at (0, 0) {$\CySpT^{m, \, 0}$};
\node (X1) [xshift=1.1*\minW] at (X0) {$\CySpT^{m, \, 1}$};
\node (X2) [xshift=1.1*\minW] at (X1) {$\CySpT^{m, \, 2}$};
\node (X3) [xshift=1.1*\minW] at (X2) {$\CySpT^{m, \, 3}$};
\node (X0h) [yshift=-\minH] at (X0) {$\CySpT_h^{m, \, 0}$};
\node (X1h) [yshift=-\minH] at (X1) {$\CySpT_h^{m, \, 1}$};
\node (X2h) [yshift=-\minH] at (X2) {$\CySpT_h^{m, \, 2}$};
\node (X3h) [yshift=-\minH] at (X3) {$\CySpT_h^{m, \, 3}$};
\node (R) [xshift=-0.6*\minW] at (X0) {$0$};
\node (Z) [xshift=0.6*\minW] at (X3) {$0$};
\node (Rh) [yshift=-\minH] at (R) {$0$};
\node (Zh) [yshift=-\minH] at (Z) {$0$};
\path[draw, arr] (X0) -- (X1) node [midway] {$\gradtm$};
\path[draw, arr] (X1) -- (X2) node [midway] {$\curltm$};
\path[draw, arr] (X2) -- (X3) node [midway] {$\divergtm$};
\path[draw, arr] (R) -- (X0);
\path[draw, arr] (X3) -- (Z);
\path[draw, arr] (X0h) -- (X1h) node [midway] {$\gradtm$};
\path[draw, arr] (X1h) -- (X2h) node [midway] {$\curltm$};
\path[draw, arr] (X2h) -- (X3h) node [midway] {$\divergtm$};
\path[draw, arr] (Rh) -- (X0h);
\path[draw, arr] (X3h) -- (Zh);
\path[draw, arr] (X0) -- (X0h) node [midway, anchor=west, yshift=0.8mm] {$\Pibr^0$};
\path[draw, arr] (X1) -- (X1h) node [midway, anchor=west, yshift=0.8mm] {$\Pibr^1$};
\path[draw, arr] (X2) -- (X2h) node [midway, anchor=west, yshift=0.8mm] {$\Pibr^2$};
\path[draw, arr] (X3) -- (X3h) node [midway, anchor=west, yshift=0.8mm] {$\Pibr^3$};
\end{tikzpicture}
\end{split}
\end{equation}
The discrete spaces in~\eqref{eq:cyl_sequence} are built using
sequences of two-dimensional Cartesian spaces defined on
$\srfdomflat$. We therefore start introducing the standard two-dimensional Cartesian spaces with the corresponding discrete
counterparts. We first introduce the basic definitions
and the necessary assumptions on the discrete spaces, postponing the
explicit definition of these to the following section. 
Let the standard spaces of square-integrable functions in Cartesian
coordinates be defined as 
\begin{equation*}
\begin{alignedat}{2}
&\Lsp^2(\srfdomflat) &&= \parG{u: \int_\srfdomflat
	u^2\,\dd\rho \dd z < \infty}~, \\ 
&\Lsp^2(\srfdomflat;\, \R^2) &&= \parG{\uv:
	\int_\srfdomflat\norm{\uv}_e^2\, \dd\rho \dd z < \infty}~, 
\end{alignedat}
\end{equation*}
along with the associated Hilbert spaces $\Hsp^s$, which is the space
of functions in $\Lsp^2$ such that the weak derivatives up to order
$s$ belong to $\Lsp^2$. We also define:
\begin{equation}
\label{eq:cart_sps}
\begin{alignedat}{3}
&\CaSpB^{0} && = \Hsp^1(\srfdomflat)&& = \parG{u \in
	\Lsp^2(\srfdomflat): \gradt u \in \Lsp^2(S)},\\ 
&\CaSpB^{1} && = \Hsp(\curltsc;\,\srfdomflat) &&= \parG{\uv \in
	\Lsp^2(\srfdomflat;\,\R^2): \curltsc \uv \in \Lsp^2(\srfdomflat)},\\ 
&\CaSpB^{1*} && = \Hsp(\divergt;\,\srfdomflat)&& = \parG{\uv \in
	\Lsp^2(\srfdomflat;\,\R^2): \divergt \uv \in \Lsp^2(\srfdomflat)},
\\ 
&\CaSpB^{2} && = \Lsp^2(\srfdomflat)&& = \parG{u:
	\int_{\srfdomflat} u^2 < \infty}. 
\end{alignedat}
\end{equation}
Here, $\gradt$ and $\divergt$ are the standard gradient and divergence
in Cartesian coordinates
\begin{equation*}
\gradt u = \begin{bmatrix}\pd{\rho}u\\ \pd{z}u\end{bmatrix}~, \qquad
\divergt\parT{\begin{bmatrix}u_\rho\\u_z\end{bmatrix}} =
\pd{\rho}{u_\rho} + \pd{z}{u_z} 
\end{equation*}
and  the scalar curl is defined as
\begin{equation*}
\curltsc\parT{\begin{bmatrix}u_\rho\\u_z\end{bmatrix}} =
\pd{\rho}{u_z} - \pd{z}{u_\rho}~. 
\end{equation*}
For functions in $\CaSpB^{0}$, also the following operator is well-defined: 
\begin{equation*}
\renewcommand\arraystretch{1.2}
\rott(u) = \begin{bmatrix}\phantom{-}\pd{z}{u} \\
-\pd{\rho}{u}\end{bmatrix} = \mat{P} \parT{\gradt(u)}~, 
\end{equation*}
where
\begin{equation*}
\mat{P} = \begin{bmatrix} 0 & 1 \\ -1 & 0\end{bmatrix}~.
\end{equation*}
We will use the word rotor, or perpendicular gradient, to indicate
$\rott$ and to avoid the confusion with the three-dimensional
$\curlt$. For error estimates we will also need the the spaces
$\Hsp^s(\curlt; \, S)$, which is the space of functions in
$\Hsp^s(S;\, \R^3)$ such that their curl belongs to $\Hsp^s(S;\,\R^3)$
and $\Hsp^s(\divergt; \, S)$, which is the space of functions in
$\Hsp^s(S;\, \R^3)$ such that their divergence belongs to
$\Hsp^s(S)$. 

We assume that $\CaSpB_h^0$, $\CaSpB_h^1$, $\CaSpB_h^{1*}$ and
$\CaSpB_h^2$ are conforming discretizations of~\eqref{eq:cart_sps} and
that, with a sequence of $\Lsp^2$-stable projectors, make the
following diagrams commutative:
\begin{equation}
\label{eq:2D_curl_comp}
\begin{split}
\tikzsetnextfilename{de_rham/de_rham_2d_cart_curl}
\begin{tikzpicture}[commutative diagrams/every diagram,
arr/.style={commutative diagrams/.cd,
	every arrow,
	every label}]
\def\minW{25mm}
\def\minH{15mm}
\def\hShift{10mm}
\def\vShiftA{20mm}
\def\vShiftB{4mm}
\def\auxA{1mm}
\def\auxB{4mm}
%
\node (X0) [] at (0, 0) {$\CaSpB^0$};
\node (X1) [xshift=1.1*\minW] at (X0) {$\CaSpB^1$};
\node (X2) [xshift=1.1*\minW] at (X1) {$\CaSpB^2$};
%
\node (X0h) [yshift=-\minH] at (X0) {$\CaSpB_h^{0}$};
\node (X1h) [yshift=-\minH] at (X1) {$\CaSpB_h^{1}$};
\node (X2h) [yshift=-\minH] at (X2) {$\CaSpB_h^{2}$};
%
\node (R) [xshift=-0.6*\minW] at (X0) {$\R$};
\node (Z) [xshift=0.6*\minW] at (X2) {$0$};
\node (Rh) [yshift=-\minH] at (R) {$\R$};
\node (Zh) [yshift=-\minH] at (Z) {$0$};
\path[draw, arr] (X0) -- (X1) node [midway] {$\gradt$};
\path[draw, arr] (X1) -- (X2) node [midway] {$\curltsc$};
\path[draw, arr] (R) -- (X0);
\path[draw, arr] (X2) -- (Z);
\path[draw, arr] (X0h) -- (X1h) node [midway] {$\gradt$};
\path[draw, arr] (X1h) -- (X2h) node [midway] {$\curltsc$};
\path[draw, arr] (Rh) -- (X0h);
\path[draw, arr] (X2h) -- (Zh);
\path[draw, arr] (X0) -- (X0h) node [midway, anchor=west, yshift=0.8mm] {$\Pi^0$};
\path[draw, arr] (X1) -- (X1h) node [midway, anchor=west, yshift=0.8mm] {$\Pi^1$};
\path[draw, arr] (X2) -- (X2h) node [midway, anchor=west, yshift=0.8mm] {$\Pi^2$};
\end{tikzpicture}
\end{split}
\end{equation}
\begin{equation}
\label{eq:2D_div_comp}
\begin{split}
\tikzsetnextfilename{de_rham/de_rham_2d_cart_div}
\begin{tikzpicture}[commutative diagrams/every diagram,
arr/.style={commutative diagrams/.cd,
	every arrow,
	every label}]
\def\minW{25mm}
\def\minH{15mm}
\def\hShift{10mm}
\def\vShiftA{20mm}
\def\vShiftB{4mm}
\def\auxA{1mm}
\def\auxB{4mm}
%
\node (X0) [] at (0, 0) {$\CaSpB^0$};
\node (X1) [xshift=1.1*\minW] at (X0) {$\CaSpB^{1*}$};
\node (X2) [xshift=1.1*\minW] at (X1) {$\CaSpB^2$};
%
\node (X0h) [yshift=-\minH] at (X0) {$\CaSpB_h^{0}$};
\node (X1h) [yshift=-\minH] at (X1) {$\CaSpB_h^{1*}$};
\node (X2h) [yshift=-\minH] at (X2) {$\CaSpB_h^{2}$};
%
\node (R) [xshift=-0.6*\minW] at (X0) {$\R$};
\node (Z) [xshift=0.6*\minW] at (X2) {$0$};
\node (Rh) [yshift=-\minH] at (R) {$\R$};
\node (Zh) [yshift=-\minH] at (Z) {$0$};
\path[draw, arr] (X0) -- (X1) node [midway] {$\rott$};
\path[draw, arr] (X1) -- (X2) node [midway] {$\divergt$};
\path[draw, arr] (R) -- (X0);
\path[draw, arr] (X2) -- (Z);
\path[draw, arr] (X0h) -- (X1h) node [midway] {$\rott$};
\path[draw, arr] (X1h) -- (X2h) node [midway] {$\divergt$};
\path[draw, arr] (Rh) -- (X0h);
\path[draw, arr] (X2h) -- (Zh);
\path[draw, arr] (X0) -- (X0h) node [midway, anchor=west, yshift=0.8mm] {$\Pi^0$};
\path[draw, arr] (X1) -- (X1h) node [midway, anchor=west, yshift=0.8mm] {$\Pi^{1*}$};
\path[draw, arr] (X2) -- (X2h) node [midway, anchor=west, yshift=0.8mm] {$\Pi^2$};
\end{tikzpicture}
\end{split}
\end{equation}
Both the continuous and the discrete sequences
in~\eqref{eq:2D_curl_comp} and~\eqref{eq:2D_div_comp} are exact for
simply connected domains~\cite{buffa:2011,buffa:2013}. 
Moreover we assume that the following estimates hold:
\begin{equation}
\label{eq:app_est_2D_cart}
\!\!\!\begin{alignedat}{3}
& \norm{u - \Pi^0 u}_{\Hsp^1(\srfdomflat)}  &&\leq C
h^{s}\norm{u}_{\Hsp^{s+1}(\srfdomflat)}~, &&\quad u \in
\Hsp^{s+1}(\srfdomflat)~,\\ 
& \norm{\uv - \Pi^1 \uv}_{\Hsp(\curltsc;\,\srfdomflat)}&&\leq C
h^{s}\norm{\uv}_{\Hsp^{s}(\curltsc;\,\srfdomflat)}~, &&\quad \uv \in
\Hsp^{s}(\curltsc;\,\srfdomflat)~,\\ 
&\norm{\uv - \Pi^{1*} \uv}_{\Hsp(\divergt;\,\srfdomflat)} &&\leq C
h^{s}\norm{\uv}_{\Hsp^{s}(\divergt;\,\srfdomflat)}~, &&\quad \uv \in
\Hsp^{s}(\divergt;\,\srfdomflat)~,\\ 
&\norm{u - \Pi^2 u}_{\Lsp^2(\srfdomflat)} &&\leq C
h^{s}\norm{u}_{\Hsp^{s}(\srfdomflat)}~, &&\quad u \in
\Hsp^{s}(\srfdomflat).
\end{alignedat}\!
\end{equation}

In \cite{oh:2015}, the discrete Cartesian spaces were chosen as
piecewise linear for $\CaSpB_h^0$, lowest order Nedelec for
$\CaSpB_h^1$, lowest order Raviart-Thomas for $\CaSpB_h^{1*}$ and
piecewise constant for $\CaSpB_h^2$. In this work, we will use \ac{IGA} spaces
that  will be introduced
in the next section. 

\medskip
The strategy to define the discrete counterparts
of~\eqref{eq:Zm_notation} is based on the use of operators that link
the spaces in cylindrical coordinates to the standard spaces in
Cartesian coordinates. The definition of the discrete spaces and of
the projectors, together with error estimates, are then deduced from
those defined for the spaces in Cartesian coordinates. The diagram
\begin{equation}
\label{eq:gen_strategy}
\begin{split}
\tikzsetnextfilename{de_rham/de_rham_3d_cyl_gen_strat}
\begin{tikzpicture}[commutative diagrams/every diagram,
arr/.style={commutative diagrams/.cd,
	every arrow,
	every label}]
\def\minW{25mm}
\def\minH{18mm}
\def\hShift{10mm}
\def\vShiftA{20mm}
\def\vShiftB{4mm}
\def\auxA{1mm}
\def\auxB{4mm}
%
\node (X0) [] at (0, 0) {$\CySpT^{m, \, 0}$};
\node (X1) [xshift=1.1*\minW] at (X0) {$\CySpT^{m, \, 1}$};
\node (X2) [xshift=1.1*\minW] at (X1) {$\CySpT^{m, \, 2}$};
\node (X3) [xshift=1.1*\minW] at (X2) {$\CySpT^{m, \, 3}$};
\node (Y0) [yshift=-1*\minH] at (X0) {$\CaSpB^{0}$};
\node (Y1) [yshift=-1*\minH] at (X1) {$\CaSpB^{1} \times \CaSpB^{0}$};
\node (Y2) [yshift=-1*\minH] at (X2) {$\CaSpB^{1*} \times \CaSpB^{2}$};
\node (Y3) [yshift=-1*\minH] at (X3) {$\CaSpB^{2}$};
\node (Y0h) [yshift=-2*\minH] at (X0) {$\CaSpB_h^{0}$};
\node (Y1h) [yshift=-2*\minH] at (X1) {$\CaSpB_h^{1} \times \CaSpB_h^{0}$};
\node (Y2h) [yshift=-2*\minH] at (X2) {$\CaSpB_h^{1*} \times \CaSpB_h^{2}$};
\node (Y3h) [yshift=-2*\minH] at (X3) {$\CaSpB_h^{2}$};
\node (X0h) [yshift=-3*\minH] at (X0) {$\CySpT_h^{m, \, 0}$};
\node (X1h) [yshift=-3*\minH] at (X1) {$\CySpT_h^{m, \, 1}$};
\node (X2h) [yshift=-3*\minH] at (X2) {$\CySpT_h^{m, \, 2}$};
\node (X3h) [yshift=-3*\minH] at (X3) {$\CySpT_h^{m, \, 3}$};
\path[draw, arr] (X0) -- (X1) node [midway] {$\gradtm$};
\path[draw, arr] (X1) -- (X2) node [midway] {$\curltm$};
\path[draw, arr] (X2) -- (X3) node [midway] {$\divergtm$};
\path[draw, arr] (Y0) -- (Y1) node [midway] {$\vec{G}$};
\path[draw, arr] (Y1) -- (Y2) node [midway] {$\vec{C}$};
\path[draw, arr] (Y2) -- (Y3) node [midway] {$D$};
\path[draw, arr] (Y0h) -- (Y1h) node [midway] {$\vec{G}$};
\path[draw, arr] (Y1h) -- (Y2h) node [midway] {$\vec{C}$};
\path[draw, arr] (Y2h) -- (Y3h) node [midway] {$D$};
\path[draw, arr] (X0h) -- (X1h) node [midway] {$\gradtm$};
\path[draw, arr] (X1h) -- (X2h) node [midway] {$\curltm$};
\path[draw, arr] (X2h) -- (X3h) node [midway] {$\divergtm$};
\path[draw, arr] (X0) -- (Y0) node [midway, anchor=west, yshift=0.8mm] {$\eta_{m, \,0}$};
\path[draw, arr] (X1) -- (Y1) node [midway, anchor=west, yshift=0.8mm] {$\eta_{m, \,1}$};
\path[draw, arr] (X2) -- (Y2) node [midway, anchor=west, yshift=0.8mm] {$\eta_{m, \,2}$};
\path[draw, arr] (X3) -- (Y3) node [midway, anchor=west, yshift=0.8mm] {$\eta_{m, \,3}$};
\path[draw, arr] (Y0) -- (Y0h) node [midway, anchor=west, yshift=0.8mm] {$\Pi^0$};
\path[draw, arr] (Y1) -- (Y1h) node [midway, anchor=west, yshift=0.8mm] {$\Pi^1 \times \Pi^0$};
\path[draw, arr] (Y2) -- (Y2h) node [midway, anchor=west, yshift=0.8mm] {$\Pi^{1*} \times \Pi^2$};
\path[draw, arr] (Y3) -- (Y3h) node [midway, anchor=west, yshift=0.8mm] {$\Pi^2$};
\path[draw, arr] (Y0h) -- (X0h) node [midway, anchor=west, yshift=0.8mm] {$\eta_{m, \,0}^{-1}$};
\path[draw, arr] (Y1h) -- (X1h) node [midway, anchor=west, yshift=0.8mm] {$\eta_{m, \,1}^{-1}$};
\path[draw, arr] (Y2h) -- (X2h) node [midway, anchor=west, yshift=0.8mm] {$\eta_{m, \,2}^{-1}$};
\path[draw, arr] (Y3h) -- (X3h) node [midway, anchor=west, yshift=0.8mm] {$\eta_{m, \,3}^{-1}$};
\end{tikzpicture}
\end{split}
\end{equation}
summarizes the general strategy for the discretization that will be
explained in the remaining section.

The first step is the definition of a set of operators $\eta_{m, \,
	k}, ~k=0, \, 1, \, 2, \, 3$ that maps functions from $\CySpT^{m, \,
	k}$ onto the functional spaces defined in Cartesian coordinates
\begin{align*}
\eta_{m, \, 0}&: \CySpT^{m, \, 0}  \rightarrow \CaSpB^0, & u & \mapsto
\dfrac{m}{\rho} u, \\ 
\eta_{m, \, 1}&: \CySpT^{m, \, 1}  \rightarrow \CaSpB^1 \times
\CaSpB^0, & 
\uv & \mapsto 
\begin{bmatrix}
\dfrac{1}{\rho} \parT{m\,u_\rho + u_\theta}\\
\dfrac{m}{\rho}u_z\\
u_\theta
\end{bmatrix},\\
\eta_{m, \, 2} &: \CySpT^{m, \, 2}  \rightarrow \CaSpB^{1*} \times
\CaSpB^2, & 
\uv & \mapsto
\begin{bmatrix}
u_\rho\\
u_z\\
\dfrac{1}{\rho}\parT{mu_\theta - u_\rho}
\end{bmatrix}, \\
\eta_{m, \, 3} &: \CySpT^{m, \, 3} \rightarrow \CaSpB^{2}, &u &
\mapsto u~. 
\end{align*}
These operators are chosen such that $\vec{G}$, $\vec{C}$ and $D$ are
both well-defined and make the top part of the
diagram~\eqref{eq:gen_strategy} commutative.  
Note that, if $\Gamma^0 \neq \emptyset$, these operators are well-defined only on regular subspaces $\CySpTT^{m, \, k}\subset\CySpT^{m,
	\, k}$ where
\begin{equation*}
\label{eq:Zm_tilde}
\begin{aligned}
\CySpTT^{m, \, 0} & = \parG{u \in \CySpT^{m, \, 0}:~ \eta_{m, \, 0}(u)
	\in \CaSpB^0}~, \\ 
\CySpTT^{m, \, 1} & = \parG{\uv \in \CySpT^{m, \, 1}:~ \eta_{m, \,
		1}(\uv) \in \CaSpB^1 \times \CaSpB^0}~, \\ 
\CySpTT^{m, \, 2} & = \parG{\uv \in \CySpT^{m, \, 2}:~ \eta_{m, \,
		2}(\uv) \in \CaSpB^{1*} \times \CaSpB^2}~, \\ 
\CySpTT^{m, \, 3} & = \parG{u \in \CySpT^{m, \, 3}:~ \eta_{m, \, 3}(u)
	\in \CaSpB^2}~. 
\end{aligned}
\end{equation*}
Then, since proper conforming discretizations are known for the spaces
in Cartesian
coordinates~\eqref{eq:2D_curl_comp}--\eqref{eq:2D_div_comp}, the
discrete spaces $\CySpT_h^{m, \, k}$ are built from the Cartesian
discrete spaces using $\eta_{m, \, k}^{-1}$, \textit{i.e.}
\begin{equation}
\label{eq:cyl_phys_disc_sps}
\begin{aligned}
\CySpT_h^{m, \, 0} & = \parG{u_h ~|~ u_h = \eta_{m,
		\,0}^{-1}\parT{\ut_h}, ~ \ut_h \in \CaSpB_h^0}~,\\ 
\CySpT_h^{m, \, 1} & = \parG{\uv_h ~|~ \uv_h = \eta_{m, \,
		1}^{-1}\parT{\uvt_h}, ~ \uvt_h \in \CaSpB_h^1 \times
	\CaSpB_h^0}~,\\ 
\CySpT_h^{m, \, 2} & = \parG{\uv_h ~|~ \uv_h = \eta_{m, \,
		2}^{-1}\parT{\uvt_h}, ~ \uvt_h \in \CaSpB_h^{1*} \times
	\CaSpB_h^2}~,\\ 
\CySpT_h^{m, \, 3} & =\parG{u_h ~|~ u_h = \eta_{m, \,
		3}^{-1}\parT{\ut_h}, ~ \ut_h \in \CaSpB_h^2}~. 
\end{aligned}
\end{equation}
Consistently with this definition, we denote
functions defined on spaces in Cartesian coordinates with a
tilde. Moreover, since in the following discussion the $\rho$ and $z$
components will be often treated differently from the $\theta$
component, we will denote with the subscript $\rho z$ the meridian component of a function
$\uv_{\rho z} = \parTs{u_\rho, \,
	u_z}^T$. 
In order to apply the differential operators to the discrete
functions, it is  useful to define, by composition, the following
operators that act directly on the continuous spaces defined in
Cartesian coordinates: 
\begin{align}
\renewcommand\arraystretch{1.5}
\begin{split}
\gradtsm & = \eta_{m, \, 0}^{-1} \comp \gradtm: \CaSpB^0
\rightarrow \CySpT^{m, \, 1}~, \\
\gradtsm\parT{\ut} & = 
\begin{bmatrix}
\dfrac{1}{m} \parT{\ut + {\rho}\,\pd{\rho}{\ut}}\\
\dfrac{\rho}{m}\pd{z}{\ut}\\
-\ut
\end{bmatrix}~, 
\end{split} \label{eq:gradstar}\\[1.5em]
\begin{split}
\curltsm & = \eta_{m, \, 1}^{-1} \comp \curltm: \CaSpB^1
\times \CaSpB^0 \rightarrow \CySpT^{m, \, 2}~, \\ 
\curltsm \parT{\uvt} & = \begin{bmatrix}
-\ut_z - \pd{z}{\ut_\theta}\\
\ut_\rho + \pd{\rho}{\ut_\theta}\\
\dfrac{1}{m}\parT{\pd{z}\parT{\rho\,\ut_\rho} -
	\pd{\rho}\parT{\rho\,\ut_z} - \pd{z}{\ut_\theta}} 
\end{bmatrix}
\end{split}~,\label{eq:curlstar}\\[1.5em]
\begin{split}
\divergtsm & = \eta_{m, \, 2}^{-1} \comp \divergtm:
\CaSpB^{1*} \times \CaSpB^2 \rightarrow \CySpT^{m, \, 3}~, \\ 
\divergtsm \parT{\uvt} & = \pd{\rho}{\ut_\rho} - \ut_\theta
+ \pd{z}{\ut_z}~.
\end{split}\label{eq:divstar}
\end{align}
Note that, due to the choice of the spaces in Cartesian coordinates,
all the differential
operators~\eqref{eq:gradstar}--\eqref{eq:divstar} are well-defined. Moreover, since the discrete spaces in Cartesian coordinates
are conforming, we have the following result: 
\begin{lemma}
	The discrete spaces~\eqref{eq:cyl_phys_disc_sps} are
	conforming in the spaces defined in~\eqref{eq:Zm_notation}. 
\end{lemma}
\begin{proof}
	We have to show that $\CySpT_h^{m, \, k} \subset \CySpT^{m, \,
		k}$, $k=1, \, \ldots, \, 3$. 
	Due to the boundness of $\srfdomflat$, there exists $R > 0$
	such that $0 < \rho < R, ~\parT{\rho, \, z} \in \srfdomflat$. 
	Considering $\ut_h \in \CaSpB_h^0$, we have that  $\ut_h \in
	\CaSpB^0$ because of the conforming  discretization of the
	Cartesian
	complexes~\eqref{eq:2D_curl_comp}--\eqref{eq:2D_div_comp}. So
	$u_h \in \CySpT^{m, \, 0}$, in fact 
	\begin{align*}
	\norm{u_h}_\rho^2 = \norm{\rho \ut_h}_\rho^2 \leq R^3
	\,\norm{\ut_h}^2 < \infty 
	\end{align*}
	and
	\begin{align*}
	\norm{\gradtm u_h}_\rho^2 & = \norm{\gradtsm
		\ut_h}_\rho^2 \leq R \,
	\norm{\gradtsm \ut_h}^2\\ 
	& \leq 3R\, \norm{\ut_h}^2 + 2R^3 \, \norm{\gradt \ut_h}^2 <
	\infty 
	\end{align*}
	since $\ut_h \in \CaSpB^0$~. The proofs of the other cases are similar. For $\CySpT_h^{m,
		\, 1} \subset \CySpT^{m, \, 1}$, given $\uv_h \in
	\CySpT_h^{m, \, 1}$, we have that
	\begin{align*}
	\norm{\uv_h}_\rho^2 & =  \norm{\dfrac{1}{m}\parT{\rho \,
			\ut_{\rho, \, h} - \ut_{\theta, \,
				h}}}_\rho^2 +
	\norm{\dfrac{\rho}{m}\ut_{z, \,
			h}}_\rho^2 + \norm{\ut_{\theta, \,
			h}}_\rho^2\\ 
	& \leq 2 R^3\,\norm{\ut_{\rho, \, h}}^2 + R^3\,\norm{\ut_{z,
			\, h}}^2 +  3R\,\norm{\ut_{\theta, \, h}}^2 < \infty 
	\end{align*}
	and 
	\begin{align*}
	\norm{\curltm\uv_h}_\rho^2 & =
	\norm{\curltsm\uvt_h}_\rho^2
	\\ 
	& =\norm{\uvt_{\rho z, \, h} + \gradt \ut_{\theta, \,
			h}}_\rho^2 \\ 
	& + \norm{\dfrac{1}{m}\parT{-\rho\,\curltsc \uvt_{\rho z, \,
				h} - \ut_{z, \, h} - \pd{z} \ut_{\theta, \,h}}}_\rho^2 \\ 
	& \leq 6\, \norm{\uvt_{\rho z, \, h}}_\rho^2 + 4
	\,\norm{\rho\,\curltsc \uvt_{\rho z, \, h}}_\rho^2 +
	4\,\norm{\gradt \ut_{\theta, \, h}}_\rho^2\\ 
	& \leq 6R\, \norm{\uvt_{\rho z, \, h}}^2 + 4R^3\,
	\norm{\curltsc \uvt_{\rho z, \, h}}^2 + 4R\, \norm{\gradt
		\ut_{\theta, \, h}}^2 < \infty 
	\end{align*}
	since $\uvt_{\rho z, \, h} \in \CaSpB^1$ and $\ut_{\theta, \,
		h} \in \CaSpB^0$.\\ 
	For $\CySpT_h^{m, \, 2} \subset \CySpT^{m, \, 2}$, given
	$\uv_h \in \CySpT_h^{m, \, 2}$, we have  that
	\begin{align*}
	\norm{\uv_h}_\rho^2 & =  \norm{\uvt_{\rho z, \, h}}_\rho^2 +
	\norm{\dfrac{1}{m}\parT{\rho
			\,\ut_{\theta, \, h} + \ut_{\rho, \,
				h}}}_\rho^2\\ 
	& \leq 3R\,\norm{\uvt_{\rho z, \, h}}^2 +
	2R^3\,\norm{\ut_{\theta, \, h}}^2 < \infty 
	\end{align*}
	and
	\begin{align*}
	\norm{\divergtm \uv_h}_\rho^2 & = \norm{\divergtsm
		\uvt_h}_\rho^2 \leq
	2\,\norm{\divergt
		\uvt_{\rho z, \, h}}_\rho^2
	+ 2\,\norm{\ut_{\theta, \,
			h}}_\rho^2 \\ 
	& \leq 2R\,\norm{\divergt \uvt_{\rho z, \, h}}^2 +
	2R\,\norm{\ut_{\theta, \, h}}^2 < \infty 
	\end{align*}
	since  $\uvt_{\rho z, \, h} \in \CaSpB^{1*}$ and $\ut_{\theta,
		\, h} \in \CaSpB^2$.\\ 
	Finally, $\CySpT_h^{m, \, 3} \subset \CySpT^{m, \, 3}$, given
	$u_h \in \CySpT_h^{m, \, 3}$, we have that
	\begin{align*}
	\norm{u_h}_\rho^2 & = \norm{\ut_h}_\rho^2 \leq R\,
	\norm{\ut_h}^2 < \infty 
	\end{align*}
	since $\ut_h \in \CaSpB^2$.
\end{proof}
The projectors $\Pibr^{m, \, k}: \CySpTT^{m, \, k} \rightarrow
\CySpT_h^{m, \, k}$ are defined as in \cite{oh:2015}, exploiting those
used in Cartesian
coordinates~\eqref{eq:2D_curl_comp}--\eqref{eq:2D_div_comp}, that is
\begin{equation}
\label{eq:cyl_interp}
\begin{alignedat}{3}
&\Pibr^0 u && = \parT{\eta_{m, \, 0}^{-1} \comp \Pi^0
	\comp \eta_{m, \, 0}}\parT{u}~,  &&\quad u \in \CySpTT^{m,
	\, 0}~,\\ 
&\Pibr^1 \uv && = \parT{\eta_{m, \, 1}^{-1}
	\comp \parT{\Pi^1 \times \Pi^0}\comp \eta_{m, \,
		1}}\parT{\uv}~,  && \quad \uv \in \CySpTT^{m, \, 1}~,\\ 
&\Pibr^2 \uv && = \parT{\eta_{m, \, 2}^{-1}
	\comp \parT{\Pi^{1*}\times \Pi^2}\comp \eta_{m, \,
		2}}\parT{\uv}~,  && \quad \uv \in \CySpTT^{m, \, 2}~,\\ 
&\Pibr^3 u && = \parT{\eta_{m, \, 3}^{-1} \comp
	\Pi^2\comp \eta_{m, \, 3}}\parT{u} = \Pi^2\parT{u}~,
&&\quad u \in \CySpTT^{m, \, 3}~. 
\end{alignedat}
\end{equation}
The following lemma shows that they are actually projectors.
\begin{lemma}
	The interpolators \eqref{eq:cyl_interp} are projectors, that is
	\begin{alignat}{3}
	&\Pibr^0 u_h && = u_h~, && \qquad u_h \in
	\CySpT_h^0~, \label{eq:proj_p0}\\ 
	&\Pibr^1 \uv_h && = \uv_h~, && \qquad\uv_h \in
	\CySpT_h^1~,\label{eq:proj_p1}\\ 
	&\Pibr^2 \uv_h && = \uv_h~, && \qquad\uv_h \in
	\CySpT_h^2~,\label{eq:proj_p2}\\ 
	&\Pibr^3 u_h && = u_h~, && \qquad u_h \in
	\CySpT_h^3~.\label{eq:proj_p3} 
	\end{alignat}
\end{lemma}
\begin{proof}
	For \eqref{eq:proj_p3} the result is immediate since the
	projector coincides with the one defined on the Cartesian
	space. 
	Also for \eqref{eq:proj_p0}--\eqref{eq:proj_p2} the result is
	trivial and due to the definition of the spaces and to the result
	in Cartesian coordinates. Taking, for example
	\eqref{eq:proj_p0}, $u_h \in \CySpT_h^0$, we have that
	\begin{equation*}
	\Pibr^0 u_h = \Pi^0\parT{\eta_{m, \, 0}\parT{u_h}} =
	\Pi^0\parT{\ut_h} = \ut_h~, \qquad\ut_h \in \CaSpB_h^0~. 
	\end{equation*}
\end{proof}
We show that the continuous and discrete spaces, together
with the projectors, form a commutative diagram. We start proving the
commutativity of the top part of the diagram~\eqref{eq:gen_strategy}. 
\begin{lemma}
	\label{lemma:comm_top}
	The following diagram commutes:
	\begin{equation*}
	\begin{split}
	\tikzsetnextfilename{de_rham/de_rham_3d_cyl_top_tilde}
	\begin{tikzpicture}[commutative diagrams/every diagram,
	arr/.style={commutative diagrams/.cd,
		every arrow,
		every label}]
	\def\minW{25mm}
	\def\minH{18mm}
	\def\hShift{10mm}
	\def\vShiftA{20mm}
	\def\vShiftB{4mm}
	\def\auxA{1mm}
	\def\auxB{4mm}
	%
	\node (X0) [] at (0, 0) {$\CySpTT^{m, \, 0}$};
	\node (X1) [xshift=1.1*\minW] at (X0) {$\CySpTT^{m, \, 1}$};
	\node (X2) [xshift=1.1*\minW] at (X1) {$\CySpTT^{m, \, 2}$};
	\node (X3) [xshift=1.1*\minW] at (X2) {$\CySpTT^{m, \, 3}$};
	\node (Y0) [yshift=-1*\minH] at (X0) {$\CaSpB^{0}$};
	\node (Y1) [yshift=-1*\minH] at (X1) {$\CaSpB^{1} \times \CaSpB^{0}$};
	\node (Y2) [yshift=-1*\minH] at (X2) {$\CaSpB^{1*} \times \CaSpB^{2}$};
	\node (Y3) [yshift=-1*\minH] at (X3) {$\CaSpB^{2}$};
	\path[draw, arr] (X0) -- (X1) node [midway] {$\gradtm$};
	\path[draw, arr] (X1) -- (X2) node [midway] {$\curltm$};
	\path[draw, arr] (X2) -- (X3) node [midway] {$\divergtm$};
	\path[draw, arr] (Y0) -- (Y1) node [midway] {$\vec{G}$};
	\path[draw, arr] (Y1) -- (Y2) node [midway] {$\vec{C}$};
	\path[draw, arr] (Y2) -- (Y3) node [midway] {$D$};
	\path[draw, arr] (X0) -- (Y0) node [midway, anchor=west, yshift=0.8mm] {$\eta_{m, \,0}$};
	\path[draw, arr] (X1) -- (Y1) node [midway, anchor=west, yshift=0.8mm] {$\eta_{m, \,1}$};
	\path[draw, arr] (X2) -- (Y2) node [midway, anchor=west, yshift=0.8mm] {$\eta_{m, \,2}$};
	\path[draw, arr] (X3) -- (Y3) node [midway, anchor=west, yshift=0.8mm] {$\eta_{m, \,3}$};
	\end{tikzpicture}
	\end{split}
	\end{equation*}
\end{lemma}
\begin{proof}
	For the first part involving the gradient, we have that
	\begin{equation*}
	\parT{\etav_{1, \, m} \comp \gradtm} u = \parT{\mathbf{G}
		\comp \eta_{0, \, m}} u~, \qquad u \in \CySpTT^{m, \, 0}~.
	\end{equation*}
	In fact, in the left-hand side we have that
	\begin{equation*} 
	\begin{aligned}
	\gradtm u &= \parT{ \pd{\rho}{u}, \, \pd{z}{u}, -
		\dfrac{m}{\rho} u}^T~, \\ 
	\etav_{1, \, m} \comp \gradtm
	u&= \parT{\dfrac{1}{\rho}\parT{m \,\pd{\rho}{u} -
			\dfrac{m}{\rho} u}, \,\dfrac{m}{\rho} \pd{z}{u}, \,  -
		\dfrac{m}{\rho} u}\\ 
	& = \parT{\pd{\rho}\parT{\dfrac{m}{\rho} u},
		\,\pd{z}\parT{\dfrac{m}{\rho} u} , \,  - \dfrac{m}{\rho} u},  
	\end{aligned}
	\end{equation*}
	which is clearly equal to the right-hand side, since $\eta_{0,
		\, m} u =\tfrac{m}{\rho} u \in \CaSpB^0$, and belongs to
	$\CaSpB^1 \times \CaSpB^0$~. \\ 
	For the part involving the curl $\parT{\etav_{2, \, m} \comp
		\curltm} \uv = \parT{\mathbf{C} \comp \etav_{1, \, m}}
	\uv~, \uv \in \widetilde{\Zvsp}^{m, \, 1}$.	In fact, in
	the left-hand side we have that
	\begin{equation*}
	\begin{aligned}
	\curltm \uv &= \begin{bmatrix}
	-\dfrac{m}{\rho}u_z - \pd{z}{u_\theta}\\
	\dfrac{1}{\rho}\parT{\pd{\rho}\parT{\rho\,u_\theta} +
		{m}\,{u_\rho}}\\ 
	\pd{z}{u_\rho} - \pd{\rho}{u_z}
	\end{bmatrix}~, \\
	\etav_{2, \, m} \comp \curltm \uv&=
	\begin{bmatrix}
	-\dfrac{m}{\rho}u_z - \pd{z}{u_\theta}\\
	\dfrac{1}{\rho}\parT{u_\theta + {m}\,{u_\rho}} +
	\pd{\rho}u_\theta\\ 
	\dfrac{1}{\rho}\parT{m\pd{z}{u_\rho} - m\pd{\rho}{u_z} +
		\dfrac{m}{\rho}u_z + \pd{z}{u_\theta}} 
	\end{bmatrix} \\
	&= \begin{bmatrix}
	-\dfrac{m}{\rho}u_z - \pd{z}{u_\theta}\\
	\dfrac{1}{\rho}\parT{u_\theta + {m}\,{u_\rho}} +
	\pd{\rho}u_\theta\\ 
	- \pd{\rho}\parT{\dfrac{m}{\rho}u_z} +
	\pd{z}\parT{\dfrac{1}{\rho}\parT{m\,u_\rho + u_\theta}} 
	\end{bmatrix},
	\end{aligned}
	\end{equation*}
	which is equal to the right-hand side $\mathbf{C} \comp
	\etav_{1, \, m} \uv$ which belongs to $\CaSpB^{1*} \times
	\CaSpB^2$.\\ 
	Finally, for the divergence part, we have that $\parT{\eta_{3,
			\, m} \comp \divergtm} \uv = \parT{D \comp \etav_{2, \,
			m}} \uv~, \uv \in \widetilde{\Zvsp}^{m, \, 2}$. In fact,
	in the left-hand side we have  that
	\begin{equation*}
	\begin{aligned}
	\eta_{3, \, m} \comp \divergtm \uv &= \divergtm \uv = 
	\pd{\rho} u_\rho + \pd{z} u_z -
	\dfrac{1}{\rho}\parT{m\,u_\theta - u_\rho} 
	\end{aligned}
	\end{equation*}
	which belongs to $\CaSpB^2$ and is clearly equal to the right-hand side, since 
	\begin{equation*}
	\etav_{2, \, m} \uv =\parT{u_\rho, \, u_z, \,
		\dfrac{1}{\rho} \parT{m\, u_{\theta} - u_\rho}}^T~.  
	\end{equation*}
\end{proof}
We have an analogous result for also for the bottom part of the
diagram~\eqref{eq:gen_strategy}: 
\begin{lemma}
	\label{lemma:comm_bot}
	The following diagram commutes:
	\begin{equation}
	\begin{split}
	\tikzsetnextfilename{de_rham/de_rham_3d_cyl_bot}
	\begin{tikzpicture}[commutative diagrams/every diagram,
	arr/.style={commutative diagrams/.cd,
		every arrow,
		every label}]
	\def\minW{25mm}
	\def\minH{18mm}
	\def\hShift{10mm}
	\def\vShiftA{20mm}
	\def\vShiftB{4mm}
	\def\auxA{1mm}
	\def\auxB{4mm}
	%
	\node (X0) [] at (0, 0) {\phantom{$\CySpT^{m, \, 0}$}};
	\node (X1) [xshift=1.1*\minW] at (X0) {\phantom{$\CySpT^{m, \, 1}$}};
	\node (X2) [xshift=1.1*\minW] at (X1) {\phantom{$\CySpT^{m, \, 2}$}};
	\node (X3) [xshift=1.1*\minW] at (X2) {\phantom{$\CySpT^{m, \, 3}$}};
	\node (Y0h) [yshift=-0*\minH] at (X0) {$\CaSpB_h^{0}$};
	\node (Y1h) [yshift=-0*\minH] at (X1) {$\CaSpB_h^{1} \times \CaSpB_h^{0}$};
	\node (Y2h) [yshift=-0*\minH] at (X2) {$\CaSpB_h^{1*} \times \CaSpB_h^{2}$};
	\node (Y3h) [yshift=-0*\minH] at (X3) {$\CaSpB_h^{2}$};
	\node (X0h) [yshift=-1*\minH] at (X0) {$\CySpT_h^{m, \, 0}$};
	\node (X1h) [yshift=-1*\minH] at (X1) {$\CySpT_h^{m, \, 1}$};
	\node (X2h) [yshift=-1*\minH] at (X2) {$\CySpT_h^{m, \, 2}$};
	\node (X3h) [yshift=-1*\minH] at (X3) {$\CySpT_h^{m, \, 3}$};
	\path[draw, arr] (Y0h) -- (Y1h) node [midway] {$\vec{G}$};
	\path[draw, arr] (Y1h) -- (Y2h) node [midway] {$\vec{C}$};
	\path[draw, arr] (Y2h) -- (Y3h) node [midway] {$D$};
	\path[draw, arr] (X0h) -- (X1h) node [midway] {$\gradtm$};
	\path[draw, arr] (X1h) -- (X2h) node [midway] {$\curltm$};
	\path[draw, arr] (X2h) -- (X3h) node [midway] {$\divergtm$};
	\path[draw, arr] (Y0h) -- (X0h) node [midway, anchor=west, yshift=0.8mm] {$\eta_{m, \,0}^{-1}$};
	\path[draw, arr] (Y1h) -- (X1h) node [midway, anchor=west, yshift=0.8mm] {$\eta_{m, \,1}^{-1}$};
	\path[draw, arr] (Y2h) -- (X2h) node [midway, anchor=west, yshift=0.8mm] {$\eta_{m, \,2}^{-1}$};
	\path[draw, arr] (Y3h) -- (X3h) node [midway, anchor=west, yshift=0.8mm] {$\eta_{m, \,3}^{-1}$};
	\end{tikzpicture}
	\end{split}
	\end{equation}
\end{lemma}
\begin{proof}
	The proof is analogous to the one of Lemma
	\ref{lemma:comm_top}. 
\end{proof}
The next step is to prove that the middle part of the
diagram~\eqref{eq:gen_strategy}, formed by the spaces in Cartesian
coordinates and the corresponding projectors, commutes. 
\begin{lemma}
	\label{lemma:comm_mid}
	The following diagram commutes:
	\begin{equation*}
	\begin{split}
	\tikzsetnextfilename{de_rham/de_rham_3d_cyl_mid}
	\begin{tikzpicture}[commutative diagrams/every diagram,
	arr/.style={commutative diagrams/.cd,
		every arrow,
		every label}]
	\def\minW{25mm}
	\def\minH{18mm}
	\def\hShift{10mm}
	\def\vShiftA{20mm}
	\def\vShiftB{4mm}
	\def\auxA{1mm}
	\def\auxB{4mm}
	%
	\node (X0) [] at (0, 0) {\phantom{$\CySpT^{m, \, 0}$}};
	\node (X1) [xshift=1.1*\minW] at (X0) {\phantom{$\CySpT^{m, \, 1}$}};
	\node (X2) [xshift=1.1*\minW] at (X1) {\phantom{$\CySpT^{m, \, 2}$}};
	\node (X3) [xshift=1.1*\minW] at (X2) {\phantom{$\CySpT^{m, \, 3}$}};
	\node (Y0) [yshift=-0*\minH] at (X0) {$\CaSpB^{0}$};
	\node (Y1) [yshift=-0*\minH] at (X1) {$\CaSpB^{1} \times \CaSpB^{0}$};
	\node (Y2) [yshift=-0*\minH] at (X2) {$\CaSpB^{1*} \times \CaSpB^{2}$};
	\node (Y3) [yshift=-0*\minH] at (X3) {$\CaSpB^{2}$};
	\node (Y0h) [yshift=-1*\minH] at (X0) {$\CaSpB_h^{0}$};
	\node (Y1h) [yshift=-1*\minH] at (X1) {$\CaSpB_h^{1} \times \CaSpB_h^{0}$};
	\node (Y2h) [yshift=-1*\minH] at (X2) {$\CaSpB_h^{1*} \times \CaSpB_h^{2}$};
	\node (Y3h) [yshift=-1*\minH] at (X3) {$\CaSpB_h^{2}$};
	\path[draw, arr] (Y0) -- (Y1) node [midway] {$\vec{G}$};
	\path[draw, arr] (Y1) -- (Y2) node [midway] {$\vec{C}$};
	\path[draw, arr] (Y2) -- (Y3) node [midway] {$D$};
	\path[draw, arr] (Y0h) -- (Y1h) node [midway] {$\vec{G}$};
	\path[draw, arr] (Y1h) -- (Y2h) node [midway] {$\vec{C}$};
	\path[draw, arr] (Y2h) -- (Y3h) node [midway] {$D$};
	\path[draw, arr] (Y0) -- (Y0h) node [midway, anchor=west, yshift=0.8mm] {$\Pi^0$};
	\path[draw, arr] (Y1) -- (Y1h) node [midway, anchor=west, yshift=0.8mm] {$\Pi^1 \times \Pi^0$};
	\path[draw, arr] (Y2) -- (Y2h) node [midway, anchor=west, yshift=0.8mm] {$\Pi^{1*} \times \Pi^2$};
	\path[draw, arr] (Y3) -- (Y3h) node [midway, anchor=west, yshift=0.8mm] {$\Pi^2$};
	\end{tikzpicture}
	\end{split}
	\end{equation*}
\end{lemma}
\begin{proof}
	For the gradient and divergence part the result is an
	immediate consequence of the commutativity of the standard two-dimensional diagrams in Cartesian coordinates \eqref{eq:2D_curl_comp} and \eqref{eq:2D_div_comp}. 
	For the curl part, we need to show that $\mathbf{C}
	\comp \parT{\Piv^1 \times \Pi^0} = \parT{\Piv^{1*} \times
		\Pi^2} \comp \mathbf{C}$. 
	For the rotor and the curl the result is immediate:
	\begin{equation*}
	\begin{aligned}
	\rott \parT{\Pi^0 \ut_\theta} & = \Piv^{1*}\parT{\rott \ut_\theta}~, \\
	\curltsc \parT{ \Piv^1{\uvt_{\rho z}}} & = \Pi^2 \parT{
		\curltsc \uvt_{\rho z}}~. 
	\end{aligned}
	\end{equation*}
	Concerning the other components we have that
	\begin{equation*}
	\begin{aligned}
	\mat{P} \Piv^1 \parT{\uvt_{\rho z}} = \Piv^{1*} \parT{ \mat{P}
		\uvt_{\rho z}} 
	\end{aligned}
	\end{equation*}
	since $\CaSpB^{1*} = \mat{P} \CaSpB^1$ and $\CaSpB_h^{1*} =
	\mat{P} \CaSpB_h^1$. 
\end{proof}
Finally, Lemmas~\ref{lemma:comm_top}--\ref{lemma:comm_mid} lead to
the following result: 
\begin{lemma}
	The following diagram commutes:
	\begin{equation*}
	\begin{split}
	\tikzsetnextfilename{de_rham/de_rham_3d_cyl_gen_strat_tilde}
	\begin{tikzpicture}[commutative diagrams/every diagram,
	arr/.style={commutative diagrams/.cd,
		every arrow,
		every label}]
	\def\minW{25mm}
	\def\minH{18mm}
	\def\hShift{10mm}
	\def\vShiftA{20mm}
	\def\vShiftB{4mm}
	\def\auxA{1mm}
	\def\auxB{4mm}
	%
	\node (X0) [] at (0, 0) {$\CySpTT^{m, \, 0}$};
	\node (X1) [xshift=1.1*\minW] at (X0) {$\CySpTT^{m, \, 1}$};
	\node (X2) [xshift=1.1*\minW] at (X1) {$\CySpTT^{m, \, 2}$};
	\node (X3) [xshift=1.1*\minW] at (X2) {$\CySpTT^{m, \, 3}$};
	\node (Y0) [yshift=-1*\minH] at (X0) {$\CaSpB^{0}$};
	\node (Y1) [yshift=-1*\minH] at (X1) {$\CaSpB^{1} \times \CaSpB^{0}$};
	\node (Y2) [yshift=-1*\minH] at (X2) {$\CaSpB^{1*} \times \CaSpB^{2}$};
	\node (Y3) [yshift=-1*\minH] at (X3) {$\CaSpB^{2}$};
	\node (Y0h) [yshift=-2*\minH] at (X0) {$\CaSpB_h^{0}$};
	\node (Y1h) [yshift=-2*\minH] at (X1) {$\CaSpB_h^{1} \times \CaSpB_h^{0}$};
	\node (Y2h) [yshift=-2*\minH] at (X2) {$\CaSpB_h^{1*} \times \CaSpB_h^{2}$};
	\node (Y3h) [yshift=-2*\minH] at (X3) {$\CaSpB_h^{2}$};
	\node (X0h) [yshift=-3*\minH] at (X0) {$\CySpT_h^{m, \, 0}$};
	\node (X1h) [yshift=-3*\minH] at (X1) {$\CySpT_h^{m, \, 1}$};
	\node (X2h) [yshift=-3*\minH] at (X2) {$\CySpT_h^{m, \, 2}$};
	\node (X3h) [yshift=-3*\minH] at (X3) {$\CySpT_h^{m, \, 3}$};
	\path[draw, arr] (X0) -- (X1) node [midway] {$\gradtm$};
	\path[draw, arr] (X1) -- (X2) node [midway] {$\curltm$};
	\path[draw, arr] (X2) -- (X3) node [midway] {$\divergtm$};
	\path[draw, arr] (Y0) -- (Y1) node [midway] {$\vec{G}$};
	\path[draw, arr] (Y1) -- (Y2) node [midway] {$\vec{C}$};
	\path[draw, arr] (Y2) -- (Y3) node [midway] {$D$};
	\path[draw, arr] (Y0h) -- (Y1h) node [midway] {$\vec{G}$};
	\path[draw, arr] (Y1h) -- (Y2h) node [midway] {$\vec{C}$};
	\path[draw, arr] (Y2h) -- (Y3h) node [midway] {$D$};
	\path[draw, arr] (X0h) -- (X1h) node [midway] {$\gradtm$};
	\path[draw, arr] (X1h) -- (X2h) node [midway] {$\curltm$};
	\path[draw, arr] (X2h) -- (X3h) node [midway] {$\divergtm$};
	\path[draw, arr] (X0) -- (Y0) node [midway, anchor=west, yshift=0.8mm] {$\eta_{m, \,0}$};
	\path[draw, arr] (X1) -- (Y1) node [midway, anchor=west, yshift=0.8mm] {$\eta_{m, \,1}$};
	\path[draw, arr] (X2) -- (Y2) node [midway, anchor=west, yshift=0.8mm] {$\eta_{m, \,2}$};
	\path[draw, arr] (X3) -- (Y3) node [midway, anchor=west, yshift=0.8mm] {$\eta_{m, \,3}$};
	\path[draw, arr] (Y0) -- (Y0h) node [midway, anchor=west, yshift=0.8mm] {$\Pi^0$};
	\path[draw, arr] (Y1) -- (Y1h) node [midway, anchor=west, yshift=0.8mm] {$\Pi^1 \times \Pi^0$};
	\path[draw, arr] (Y2) -- (Y2h) node [midway, anchor=west, yshift=0.8mm] {$\Pi^{1*} \times \Pi^2$};
	\path[draw, arr] (Y3) -- (Y3h) node [midway, anchor=west, yshift=0.8mm] {$\Pi^2$};
	\path[draw, arr] (Y0h) -- (X0h) node [midway, anchor=west, yshift=0.8mm] {$\eta_{m, \,0}^{-1}$};
	\path[draw, arr] (Y1h) -- (X1h) node [midway, anchor=west, yshift=0.8mm] {$\eta_{m, \,1}^{-1}$};
	\path[draw, arr] (Y2h) -- (X2h) node [midway, anchor=west, yshift=0.8mm] {$\eta_{m, \,2}^{-1}$};
	\path[draw, arr] (Y3h) -- (X3h) node [midway, anchor=west, yshift=0.8mm] {$\eta_{m, \,3}^{-1}$};
	\end{tikzpicture}
	\end{split}
	\end{equation*}
\end{lemma}
\begin{proof}
	The proof is a direct consequence of
	Lemmas~\ref{lemma:comm_top}--\ref{lemma:comm_mid}.  
	For example, we have that
	\begin{align*}
	\Pivbr^2 \parT{\curltm \uv} & = \parT{\etav_{2, \, m}^{-1}
		\comp \parT{\Piv^{1*} \times \Pi^2}\comp \etav_{2, \, m}
		\comp \curltm} \uv &&\nonumber\\ 
	& = \parT{\etav_{2, \, m}^{-1} \comp \parT{\Piv^{1*} \times
			\Pi^2}\comp \mathbf{C} \comp \etav_{1, \, m} } \uv  &&
	\text{Lemma \ref{lemma:comm_top}}\nonumber\\ 
	& = \parT{\etav_{2, \, m}^{-1} \comp \mathbf{C}
		\comp \parT{\Piv^1 \times \Pi^0}\comp  \etav_{1, \, m} } \uv
	&& \text{Lemma \ref{lemma:comm_mid}}\nonumber\\ 
	& = \parT{\curltm \comp \etav_{1, \, m}^{-1}
		\comp \parT{\Piv^1 \times \Pi^0}\comp  \etav_{1, \, m} } \uv
	&& \text{Lemma \ref{lemma:comm_bot}}\nonumber\\ 
	& = \curltm \parT{\Pivbr^1 \uv} &&
	\end{align*}
\end{proof}
Another important property that can be deduced from the spaces in
Cartesian coordinates is the exactness of the discrete sequence: 
\begin{theorem}
	The discrete sequence
	\begin{equation*}
	\begin{split}
	\tikzsetnextfilename{de_rham/de_rham_3d_cyl_discr}
	\begin{tikzpicture}[commutative diagrams/every diagram,
	arr/.style={commutative diagrams/.cd,
		every arrow,
		every label}]
	\def\minW{25mm}
	\def\minH{20mm}
	\def\hShift{10mm}
	\def\vShiftA{20mm}
	\def\vShiftB{4mm}
	\def\auxA{1mm}
	\def\auxB{4mm}
	%
	\node (X0) [] at (0, 0) {$\CySpT_h^{m, \, 0}$};
	\node (X1) [xshift=1.1*\minW] at (X0) {$\CySpT_h^{m, \, 1}$};
	\node (X2) [xshift=1.1*\minW] at (X1) {$\CySpT_h^{m, \, 2}$};
	\node (X3) [xshift=1.1*\minW] at (X2) {$\CySpT_h^{m, \, 3}$};
	\node (R) [xshift=-0.6*\minW] at (X0) {$0$};
	\node (Z) [xshift=0.6*\minW] at (X3) {$0$};
	\path[draw, arr] (X0) -- (X1) node [midway] {$\gradtm$};
	\path[draw, arr] (X1) -- (X2) node [midway] {$\curltm$};
	\path[draw, arr] (X2) -- (X3) node [midway] {$\divergtm$};
	\path[draw, arr] (R) -- (X0);
	\path[draw, arr] (X3) -- (Z);
	\end{tikzpicture}
	\end{split}
	\end{equation*}
	is exact.
\end{theorem}
\begin{proof}If $u_h \in \CySpT_h^{m, \, 3}$, then $\ut_h \in
	\CaSpB_h^{2}$. So the function 
	\begin{equation*}
	\vv_h = \eta_{m, \,
		2}^{-1} \parT{\parT{\vec{0}^T, \, -\ut_h}^T} \in \CySpT_h^{m, \,
		2}
	\end{equation*} is such that $\divergt^m \vv_h = u_h$. 
	Similarly, if $\uv_h \in \CySpT_h^{m, \, 2}$ is such that
	$\divergt^m \uv_h = \vec{0}$, we have that
	\begin{equation*}
	0 = \divergtm \uv_h = \divergtsm \uvt_h =
	\divergt \uvt_{h, \, \rho z} - \ut_{h, \, \theta}~, 
	\end{equation*}
	with $\ut_{h, \, \theta} \in \CaSpB_h^{2}$. Choosing  
	\begin{equation*}
	\vv_h = \eta_{m, \, 1}^{-1} \left(\begin{bmatrix}
	\mat{P} \uvt_{h, \, \rho z}\\
	0
	\end{bmatrix}\right) \in \CySpT_h^{m, \, 1}
	\end{equation*}
	we have that
	\begin{equation*}
	\begin{aligned}
	\curltm \vv_h & = \eta_{m, \, 2}^{-1}\parT{\vec{C} \vvt_h} = \eta_{m, \, 2}^{-1}\parT{\begin{bmatrix}
	- \mat{P} \mat{P} \uvt_{h, \, \rho z} \\
	- \curltsc \parT{ \mat{P}\uvt_{h, \, \rho z}}
	\end{bmatrix}}\\
& 	= \eta_{m, \, 2}^{-1}\parT{\begin{bmatrix}
	\uvt_{h, \, \rho z} \\
	\divergt {\uvt_{h, \, \rho z}}
	\end{bmatrix}}= \eta_{m, \, 2}^{-1}\parT{\begin{bmatrix}
	\uvt_{h, \, \rho z} \\
	\ut_{h, \, \theta}
	\end{bmatrix}} = \uv_h~.
	\end{aligned}
	\end{equation*}
	If $\uv_h \in \CySpT_h^{m, \, 1}$ is such that $\curlt^m \uv_h
	= \vec{0}$, it follows that $\curltsm \uvt_h =
	\vec{0}$ and so that (see~\eqref{eq:curlstar}) $-\mat{P}
	\uvt_{h, \, \rho z} = \rott \ut_{h, \, \theta} = \mat{P}
	\gradt \ut_{h, \, \theta}$. These previous relations imply
	that the third component is identically null. Moreover, we can
	also see that $\uvt_{h, \, \rho z} = -\gradt \ut_{h, \,
		\theta}$ so, choosing  $v_h = \eta_{m, \,
		0}^{-1} \parT{-\ut_{h, \, \theta}}$ we have that $\gradt^m
	v_h = \uv_h$. 
	Finally, we see that the equation corresponding to the third
	component in $\gradt^m u_h= \vec{0}$ implies that $u_h \equiv
	0$. 
\end{proof}
In order to prove error estimates, let us define the following spaces: 
\begin{equation*}
\begin{aligned}
\Hsp_\rho^{s}(\widetilde{{\gradt}}^m) = &\Big\{ u \in
\Hsp(\gradtm):\, \eta_{m, \,0} (u) \in \Hsp_\rho^s,\,
\widetilde{{\gradt}}^m u \in \Hsp_\rho^s\Big\},\\ 
\Hsp_\rho^{s}(\widetilde{{\curlt}}^m) = &\Big\{ \uv \in
\Hsp(\curltm):\, \eta_{m, \,1} (\uv) \in \Hsp_\rho^s,  \,  
\widetilde{{\curlt}}^m \uv \in \Hsp_\rho^s\Big\},\\
\Hsp_\rho^{s}(\widetilde{{\divergt}}^m) = &\Big\{ \uv \in
\Hsp(\divergtm):\, \eta_{m, \,2} (\uv) \in \Hsp_\rho^s,
\,\widetilde{{\divergt}}^m \uv \in \Hsp_\rho^s\Big\}~, 
\end{aligned}
\end{equation*}
where
\begin{equation*}
\begin{aligned}
\widetilde{{\gradt}}^m u & = \eta_{m, \, 1} \parT{\gradtm u}~, \\
\widetilde{{\curlt}}^m \uv & = \eta_{m, \, 2} \parT{\curltm \uv}~, \\ 
\widetilde{{\divergt}}^m \uv & = \eta_{m, \, 3} \parT{\divergtm
	\uv} . 
\end{aligned}
\end{equation*}
The proof of the following estimates is analogous to the one in
\cite[Theorem 4.1]{oh:2015}.
\begin{lemma}
	\label{lemma:stab_cyl_proj}
	For $s\geq 0$, we have the error estimates
	\begin{align}
	\begin{split}
	& \norm{u - \Pibr^0\,u}_{\Hsp(\gradtm)} \leq C h^{s} \Big(
	\norm{\eta_{m, \, 0}\parT{u}}_{\Hsp_\rho^{2+s}} +
	\norm{\eta_{m, \, 1} \parT{\gradtm u} }_{\Hsp_\rho^{2+s}}
	\Big),\\  
	&\hspace{0.65\textwidth}{u \in
		\Hsp_\rho^{s+2}(\widetilde{{\gradt}}^m),} 
	\end{split} \label{eq:est_z0} \\
	\begin{split}
	&	\norm{\uv - \Pibr^1\,\uv}_{\Hsp(\curltm)} \leq C
	h^{s}  \Big( \norm{\eta_{m, \, 1}\parT{\uv}}_{\Hsp_\rho^{2+s}}
	+ \norm{\eta_{m, \, 2} \parT{\curltm \uv}
	}_{\Hsp_\rho^{2+s}} \Big), \\  
	&\hspace{0.65\textwidth} \uv \in
	\Hsp_\rho^{s+2}(\widetilde{{\curlt}}^m),  
	\end{split}\label{eq:est_z1} \\
	\begin{split}
	&	\norm{\uv - \Pibr^2\,\uv}_{\Hvsp(\divergtm)} \leq C
	h^{s}  \Big( \norm{\eta_{m, \, 2}\parT{\uv}}_{\Hsp_\rho^{2+s}}
	+ \norm{\eta_{m, \, 3} \parT{\divergtm \uv}
	}_{\Hsp_\rho^{2+s}} \Big), \\  
	&\hspace{0.65\textwidth} \uv \in
	\Hsp_\rho^{s+2}(\widetilde{{\divergt}}^m),  
	\end{split}\label{eq:est_z2} \\
	\begin{split}
	&	\norm{u - \Pibr^3\,u}_{\Lsp_\rho^2} \leq C h^{s}
	\norm{u}_{\Hsp_\rho^{2+s}} \hspace{0.29\textwidth} u \in
	\Hsp_\rho^{s+2}. 
	\end{split}\label{eq:est_z3}
	\end{align}
\end{lemma}
\begin{proof}
	Since the domain is bounded we have that $0<\rho <
	R$. Moreover, the continuous embedding
	$\Hsp_\rho^{s+2} \hookrightarrow \Hsp^{s}$ holds \cite{mercier:1982}
	and so, for $s\geq 2$ we have that
	$\Hsp_\rho^{s}(\widetilde{{\gradt}}^m) \subset \CySpTT^{m, \,
		0}$, $\Hsp_\rho^{s}(\widetilde{{\curlt}}^m) \subset
	\CySpTT^{m, \, 1}$, $\Hsp_\rho^{s}(\widetilde{{\divergt}}^m)
	\subset \CySpTT^{m, \, 2}$ and $\Hsp_\rho^{s} \subset
	\CySpTT^{m, \, 3}$. The procedure to derive the estimates is
	analogous for all the cases and can be summarized as follows,
	for $k=1, 2, 3$: 
	\begin{align*}
	\norm{u - \Pibr^k u}_\rho & = \norm{u - \parT{\eta_{m, \,
				k}^{-1} \comp \Pi^k \comp \eta_{m,
				\, k}} u}_\rho \\ 
	& = \norm{\parT{\eta_{m, \, k}^{-1}\comp \eta_{m, \, k}} u
		- \parT{\eta_{m, \, k}^{-1} \comp \Pi^k \comp \eta_{m, \, k}}
		u}_\rho \\ 
	& = \norm{\eta_{m, \, k}^{-1}\parT{\eta_{m, \, k}(u) -
			\Pi^k \parT{\eta_{m, \, k}(u)}}}_\rho\\ 
	& \leq C \, \norm{\eta_{m, \, k}^{-1}\parT{\eta_{m, \, k}(u) -
			\Pi^k \parT{\eta_{m, \, k}(u)}}} \\ 
	& \leq C \, \norm{\eta_{m, \, k}(u) - \Pi^k \parT{\eta_{m, \,
				k}(u)}} \\ 
	& \leq C \, h^s \, \norm{\eta_{m, \, k}(u)}_{\Hsp^s}\\
	& \leq C \, h^s \, \norm{\eta_{m, \,
			k}(u)}_{\Hsp_\rho^{s+2}}~, 
	\end{align*}
	where the constant $C$ is different for each inequality, but
	is independent of $u$ and can depends on $h$ only through the
	ratio between the biggest and the smallest element $h_{max} /
	h_{min}$. 
	The estimates for the terms involving the differential
	operators are reduced to the previous ones exploiting the
	commutativity of the projectors. 
	Consider $u \in \Hsp_\rho^{s+2}$, we have that
	\begin{align*}
	\norm{u - \Pibr^3 u}_\rho & \leq
	\norm{\sqrt{\dfrac{R}{\rho}}\,\parT{u
			- \Pi^2 u}}_\rho = \sqrt{R}\norm{u
		- \Pi^2 u} &&\\ 
	& \leq \sqrt{R}\,C\,h^s \norm{u}_{\Hsp^s} \leq  \widetilde{C}
	h^s \norm{u}_{\Hsp_\rho^{s+2}}~. && 
	\end{align*}
	This proves~\eqref{eq:est_z3}. \\
	Considering a function $ u \in
	\Hsp_\rho^{s+2}(\widetilde{{\gradt}}^m)$, it follows 
	\begin{align*}
	\norm{u - \Pibr^0 u}_\rho & = \norm{u - \dfrac{\rho}{m}
		\Pi^0 \parT{\dfrac{m}{\rho}u}}_\rho
	\\ 
	& = \norm{\dfrac{\rho}{m}\parT{\dfrac{m}{\rho}u} -
		\Pi^0 \parT{\dfrac{m}{\rho}u}}_\rho \\ 
	& \leq R \norm{\dfrac{m}{\rho}u -
		\Pi^0 \parT{\dfrac{m}{\rho}u}}_\rho \leq C h^s
	\norm{\dfrac{m}{\rho}u}_{\Hsp_\rho^{s+2}}. 
	\end{align*}
	So, we have that
	\begin{equation}
	\label{eq:app_pi0}
	\norm{u - \Pibr^0 u}_\rho \leq C h^s
	\norm{\dfrac{m}{\rho}u}_{\Hsp_\rho^{s+2}} = C h^s
	\norm{\eta_{m, \, 0} (u)}_{\Hsp_\rho^{s+2}}~. 
	\end{equation}
	Consider now a function $\uv \in
	\Hsp_\rho^{s+2}(\widetilde{{\divergt}}^m)$, the
	estimate~\eqref{eq:est_z2} involves the norm 
	\begin{equation*}
	\norm{\uv}_{\Hsp_\rho(\divergtm)}^2 =
	\norm{\uv}_{\Lsp_\rho^2}^2 +
	\norm{\divergtm\uv}_{\Lsp_\rho^2}^2. 
	\end{equation*}
	For the first term we have that
	\begin{align*}
	\norm{\uv - \Pibr^2 \uv}_\rho^2 & = \norm{\uv_{\rho z} -
		\Pi^{1*} \uv_{\rho
			z}}_\rho^2 \\ 
	& + \norm{ u_\theta -
		\dfrac{\rho}{m}\Pi^2 \parT{\dfrac{m\,u_\theta - u_\rho}{\rho}} -
		\dfrac{1}{m}\parQ{\Pi^{1*} \uv_{\rho z}}_\rho }_\rho^2 \\ 
	\intertext{adding and subtracting $\tfrac{u_\rho}{m}$ in the
		second term we obtain} 
	& \leq \norm{\uv_{\rho z} - \Pi^{1*} \uv_{\rho z}}_\rho^2 \\
	& + 2\,\norm{\dfrac{\rho}{m}\parT{ \dfrac{m\,u_\theta -
				u_\rho}{\rho} - \Pi^2 \parT{\dfrac{m\, u_\theta -
					u_\rho}{\rho}}}}_\rho^2 \\ 
	& + 2\,\norm{\dfrac{1}{m}\parT{u_\rho- \parQ{\Pi^{1*}
				\uv_{\rho z}}_\rho }}_\rho^2 \\ 
	& \leq 3\,\norm{\uv_{\rho z} - \Pi^{1*} \uv_{\rho z}}_\rho^2
	+R^2\,\norm{\parT{ \ut_\theta - \Pi^2 \ut_\theta}}_\rho^2  \\ 
	& \leq C h^{2s} \parT{\norm{\uv}_{\Hsp_\rho^{s+2}}^2 +
		\norm{\ut}_{\Hsp_\rho^{s+2}}^2} =C h^{2s}\norm{\eta_{m, \,
			2}(\uv)}_{\Hsp_\rho^{s+2}}^2  . 
	\end{align*}
	We are left with
	\begin{equation}
	\label{eq:app_p2}
	\norm{\uv - \Pibr^2 \uv}_{\Lsp_\rho^2} \leq C h^s
	\norm{\eta_{m, \, 2}(\uv)}_{\Hsp_\rho^{s+2}}. 
	\end{equation}
	The estimate for $\norm{\divergt^m \uv}_\rho^2$ follows from
	the commutativity   of the projectors
	and~\eqref{eq:est_z3}: 
	\begin{align*}
	\norm{\divergtm \parT{\uv - \Pibr^2 \uv}}_\rho =
	\norm{\divergtm\uv - \Pibr^3 \parT{\divergtm\uv}}_\rho
	\leq C h^s \norm{\divergtm\uv}_{\Hsp_\rho^{s+2}}. 
	\end{align*}
	This concludes the proof of~\eqref{eq:est_z2}. The
	estimate~\eqref{eq:est_z1} is proven analogously,
	considering the norm
	\begin{equation*}
	\norm{\uv}_{\Hsp_\rho(\curltm)}^2 =
	\norm{\uv}_{\Lsp_\rho^2}^2 +
	\norm{\curltm\uv}_{\Lsp_\rho^2}^2. 
	\end{equation*}
	For the first term we have that
	\begin{align*}
	\norm{\uv - \Pibr^1 \uv}_\rho^2 & = \norm{u_\theta - \Pi^0
		u_\theta}_\rho^2 \\ 
	& + \norm{\uv_{\rho z} -
		\dfrac{\rho}{m}\Pi^1 \parT{\widetilde{\uv}_{\rho z}} - \dfrac{1}{m} 
		\begin{bmatrix}
		\Pi^0 u_\theta\\
		0
		\end{bmatrix}
	}_\rho^2 \\
	\intertext{adding and subtracting $\parT{\tfrac{u_\theta}{m},
			\, 0}^T$ in the second term we obtain} 
	& \leq \norm{u_\theta - \Pi^0 u\theta}_\rho^2 \\
	& + 2\norm{\dfrac{\rho}{m}\parT{ \widetilde{\uv}_{\rho z} -
			\Pi^1 \parT{\widetilde{\uv}_{\rho z}}}}_\rho^2 
	+ 2\norm{u_\theta-\Pi^0 u_\theta}_\rho^2 \\
	& \leq C h^{2s} \parT{\norm{u_\theta}_{\Hsp_\rho^{s+2}}^2 +
		\norm{\widetilde{\uv}_{\rho z}}_{\Hsp_\rho^{s+2}}^2} =  C
	h^{2s}\norm{\eta_{m, \, 1}(\uv)}_{\Hsp_\rho^{s+2}}^2. 
	\end{align*}
	So we have that
	\begin{equation*}
	\label{eq:app_p1}
	\norm{\uv - \Pibr^1 \uv}_{\Lsp_\rho^2} \leq C h^s
	\norm{\eta_{m, \, 1}(\uv)}_{\Hsp_\rho^{s+2}}. 
	\end{equation*}
	The estimate for $\norm{\curlt^m \uv}_\rho^2$ follows from
	the commutativity   of the projectors
	and~\eqref{eq:app_p2}: 
	\begin{align*}
	\norm{\curltm \parT{\uv - \Pibr^1 \uv}}_\rho & =
	\norm{\curltm\uv
		-
		\Pibr^2 \parT{\curltm\uv}}_\rho
	\\ 
	& \leq C h^s \norm{\eta_{m, \,
			2}\parT{\curltm\uv}}_{\Hsp_\rho^{s+2}}~. 
	\end{align*}
	This conclude the proof of~\eqref{eq:est_z1}. Consider a
	function $ u \in \Hsp_\rho^{s+2}(\widetilde{{\gradt}}^m)$, the
	norm involved in~\eqref{eq:est_z0} is
	\begin{equation*}
	\norm{u}_{\Hsp_\rho(\gradtm)}^2 = \norm{u}_{\Lsp_\rho^2}^2
	+ \norm{\gradtm u}_{\Lsp_\rho^2}^2. 
	\end{equation*}
	For the first term, we already have the estimate given
	by~\eqref{eq:app_pi0}. The estimate for the second term
	follows from the commutativity property of the projectors
	and~\eqref{eq:app_p2}: 
	\begin{align*}
	\norm{\gradtm \parT{u - \Pibr^0 u}}_\rho & =
	\norm{\gradtm u - \Pibr^1 \parT{\gradtm u}}_\rho \\
	& \leq C h^s \norm{\eta_{m, \, 1}\parT{\gradtm u}}_{\Hsp_\rho^{s+2}}.
	\end{align*}
	This concludes the proof of~\eqref{eq:est_z0}.
\end{proof}

\section{\acl{IGA}}
\label{sec:iga}
In this section we briefly introduce the basic concepts of \acl{IGA}
following \cite{higueras:2016}. We will start with the definition of
\acp{B-spline} on a reference two-dimensional domain $\igasrfpar$ in
the univariate case and then we introduce its extension, via tensor
product, to the multivariate case. 

Considering an exact de Rham
complex of continuous spaces defined in the reference domain, we then
define the corresponding conforming discrete \ac{B-spline} spaces and
projectors and we show that they form a commuting diagram. 
We will consider a regular parametrization of the two-dimensional
physical domain, which will be the cross-section of our axisymmetric
domain $\srfdomflat$, described by a \ac{B-spline} or by a \ac{NURBS}
surface. All the results valid on the parametric domain can be
extended to the case of the physical one and error estimates
like~\eqref{eq:app_est_2D_cart} hold.\\ 

We start defining the so-called knot vector $\Xiv$, which is a sequence of
ordered real numbers that we assume, without loss of generality, bounded by
$0$ and $1$. In this work we will consider only open knot vectors,
which are characterized by the fact that the first $p+1$ knots are
equal to $0$ and the last $p+1$ knots are equal to $1$, \textit{i.e.} 
\begin{equation*}
\label{eq:def_Xi}
\Xiv = \parG{\xi_1 = \ldots = \xi_{p+1} < \ldots < \xi_{n+1} = \ldots
	= \xi_{n+p+1}}~, 
\end{equation*}
where $p$ is the degree and $n$ is the number of the \ac{B-spline}
polynomials. 
\ac{B-spline} polynomials are the fundamental functions used to define
our finite-dimensional spaces and can be defined recursively using the
well known Cox-DeBoor formula: starting from piecewise constant
polynomials $\parT{p = 0}$ 
\begin{equation*}
\Bh_i^0\parT{\zeta} = \left\{ 
\begin{aligned}
1 & \qquad \xi_i \leq \zeta < \xi_{i+1} \\
0 & \qquad \text{otherwise}
\end{aligned}
\right.~,
\end{equation*}
the higher degree polynomials $\parT{p \geq 1}$ are defined by
\begin{equation*}
\Bh_i^p\parT{\zeta} = \dfrac{\zeta - \xi_i}{\xi_{i+p} - \xi_i}
\Bh_i^{p-1}\parT{\zeta} + \dfrac{\xi_{i+p+1} - \zeta}{\xi_{i+p+1} -
	\xi_{i+1}} \Bh_{i+1}^{p-1}\parT{\zeta}~, 
\end{equation*}
with the convention that $0/0$ is equal to $0$.
This formula generates a set of $n$ \ac{B-spline} which  has many
favorable properties.  In particular, these functions are non-negative,   form a partition of
unity and  have local support. Moreover, the support the $i$-th
\ac{B-spline} is contained in the interval $\parQ{\xi_i, \,
	\xi_{i+p+1}}$, so the size of the support is reduced by knot
repetitions (see Figure~\ref{fig:bspl_uni}): 
\begin{equation*}
\Bh_{i}^{p}\parT{\zeta}=0~, \qquad \zeta \notin \parQ{\xi_i, \,
	\xi_{i+p+1}}~. 
\end{equation*}
Conversely, in each interval $\parQ{\xi_j, \, \xi_{j+1}}$ there are
exactly $p+1$ \acp{B-spline} which are different from $0$: 
\begin{equation*}
\Bh_{i}^{p}\parT{\zeta}=0~, \qquad \zeta \in \parQ{\xi_j, \, \xi_{j+1}},
\, i \notin \parG{j, \, j-1, \, \ldots, \, j-p}~. 
\end{equation*}
It is possible to describe the knot vector $\Xiv$ using other two
vectors: a vector containing the knots without repetition, that we
indicate with $\zetav \in \R^\ell$, and a vector containing the number
of times each knot is repeated $1 \leq r_i \leq p, \, 2 \leq i \leq
\ell-1$, with $r_1 = r_\ell = p+1$. The number $\alpha_i = p - r_i$
denotes the regularity of the \ac{B-spline} function at the knot
$\zeta_i$. 
In analogy to the standard \ac{FEM}, we can use $\zetav$
to define elements of a mesh with the corresponding mesh size $h_i =
\zeta_{i+1} - \zeta_i, ~1 \leq i \leq \ell - 1$. We say that the
partition defined by $\zetav$ is locally quasi uniform if there exists
a constant $\eta \geq 1$ such that 
\begin{equation*}
\eta^{-1} \leq \dfrac{h_i}{h_{i+1}} \leq \eta, \qquad 1\leq i
\leq \ell-2. 
\end{equation*}
We can now define the spline spaces \cite{buffa:2011, buffa:2010} as
\begin{equation*}
S^{p}_{\alphav}\parT{\zetav} = \spant \parG{\Bh_i^p, \,
	i=1, \, \ldots, \, n}~. 
\end{equation*}
Note that spline space can be completely characterized either by the
knot vector $\Xiv$ or by the degree $p$, the mesh $\zetav$ and the
regularity $\alphav$ (or the knot repetitions).  
Figure~\ref{fig:bspl_uni} shows two sets of quadratic \ac{B-spline}
basis functions generated by two different knot vectors with the same
elements but different regularity. 
\begin{figure}[!htb]
	\centering
	\includegraphics[width=0.42\textwidth]{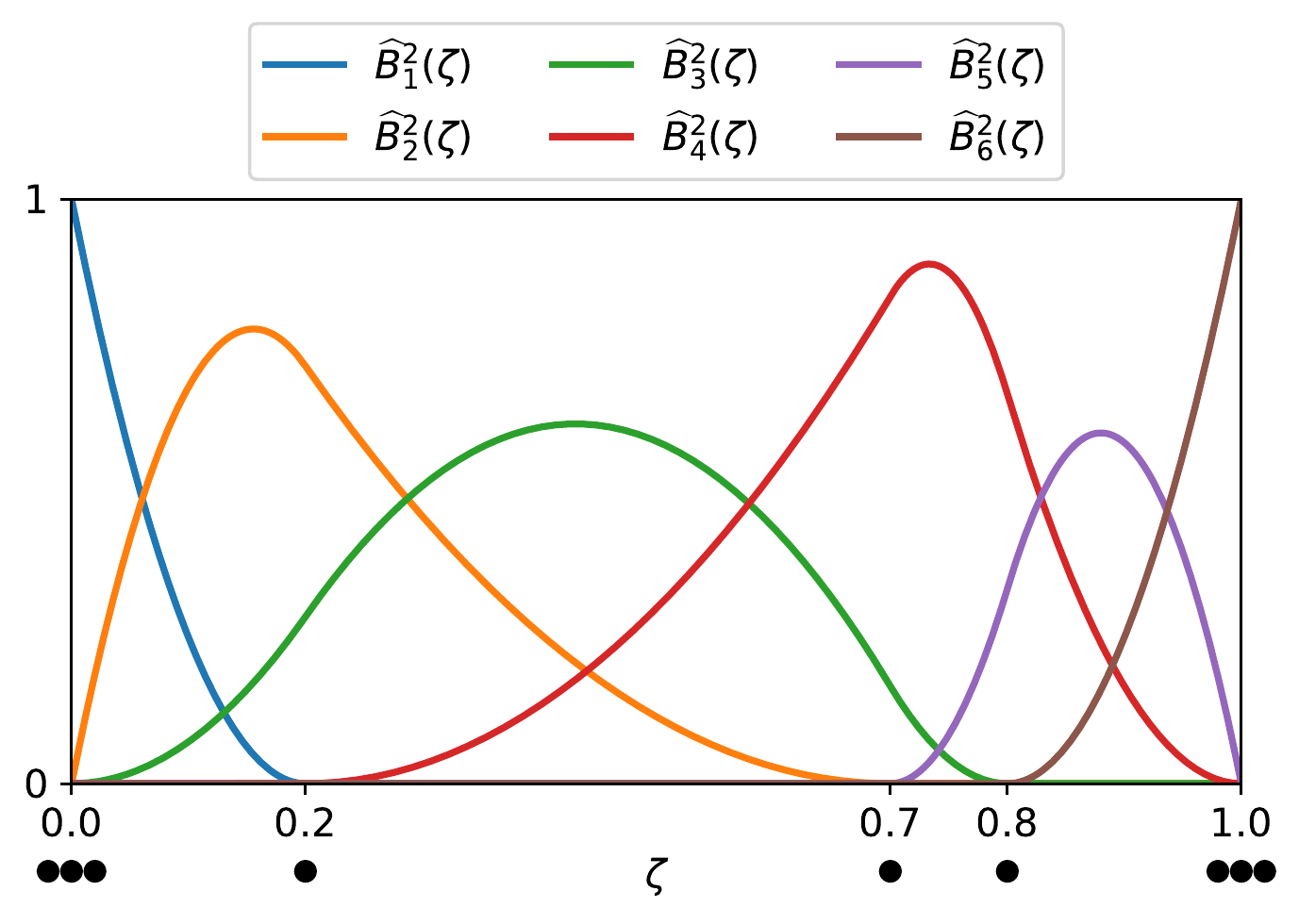}~~ 
	\includegraphics[width=0.42\textwidth]{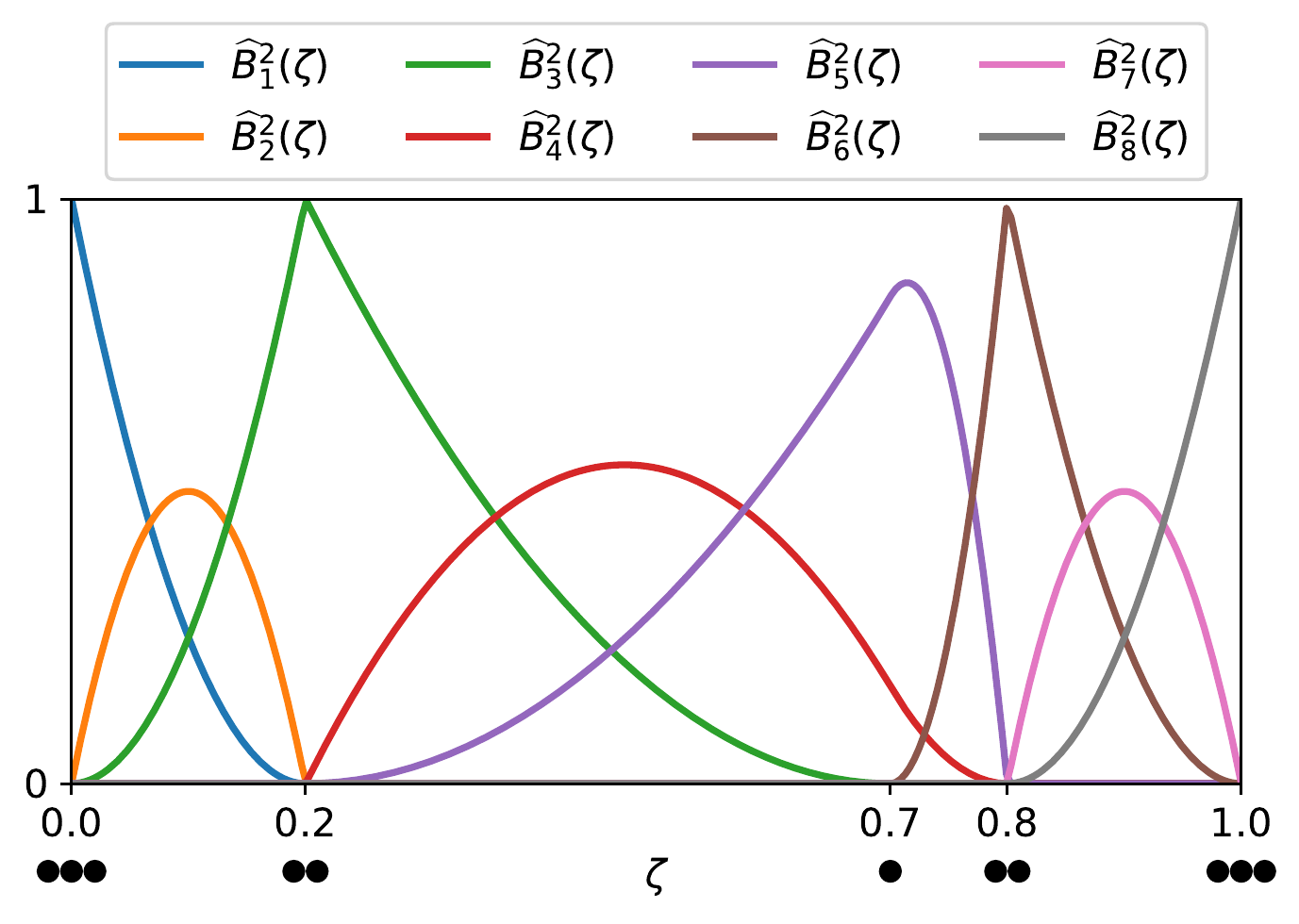}
	\caption{Two set of quadratic \ac{B-spline} basis functions
		generated by knot vectors including the same points with different multiplicities.
		The dots below the
		$\zeta$-axis indicate the knot repetitions. 
		Note that the support
		of the less regular \acp{B-spline} (right) is smaller that
		the ones with higher regularity (left).} 
	\label{fig:bspl_uni}
\end{figure}
The multivariate spaces are simply defined via tensor product. We will explicitly introduce the quantities in the
two-dimensional case, since it is the most relevant  for this work. 
Let $\Xiv_1$ and $\Xiv_2$ be   two locally quasi uniform open knot
vectors corresponding to parametric directions. The associated
shape-regular Bezier mesh on $\igasrfpar = \parT{0, \, 1}^2$ is given by 
\begin{equation*}
\Qch_h = \parG{{Q} = \parT{\zeta_{i}, \, \zeta_{i+1}}
	\times \parT{\zeta_{j}, \, \zeta_{j+1}}, \, 1 \leq i \leq
	\ell_1, ~ 1 \leq j \leq \ell_2}, 
\end{equation*}
where $h = \max\parG{\diam({Q}), \, {Q} \in \Qch_h}$ is the global
mesh size. We will denote the coarsest mesh  $\Qch_0 $ with the subscript $0$. 
From the univariate \ac{B-spline} basis functions $\Bh_{i, \, d}$
defined by $\Xiv_d$, $d = 1, 2$, we define the tensor product basis
functions on $\igasrfpar$: 
\begin{equation*}
\Bh_{ij}^{p_1, \, p_2} =\Bh_{i, \, 1}^{p_1} \otimes
\Bh_{j, \, 2}^{p_2}, \qquad 1 \leq i \leq n_1, \, 1 \leq j
\leq n_2 
\end{equation*}
and the spline space is analogously defined as:
\begin{equation*}
S^{p_1, \, p_2}_{\alphav_1, \, \alphav_2}\parT{\Qch_h} =
\spant \parG{\Bh_{ij}^{p_1, \, p_2}, \, 1 \leq i \leq n_1, \,
	1 \leq j \leq n_2}~. 
\end{equation*}
We also define a   generalization of \acp{B-spline}, called
\ac{NURBS}, which is particularly useful to describe the geometry
since it allows to exactly represent conic sections. Given set of
weights $w_{ij} \geq 0, \, 1 \leq i \leq n_1, \,  1 \leq j \leq n_2$
thaat sum to one, \ac{NURBS} are defined by
\begin{equation*}
\widehat{N}_{ij}^{p_1, \, p_2} = \dfrac{w_{ij}\Bh_{ij}^{p_1,
		\, p_2}}{\sum_{k, \ell}w_{k\ell}\Bh_{k\ell}^{p_1, \,
		p_2}}. 
\end{equation*}
Gathering the weights in a vector $\veca{W}$ we define the associated
\ac{NURBS} space: 
\begin{equation*}
N_{\alphav_1, \, \alphav_2}^{p_1, \, p_2}\parT{\Qch_h,
	\veca{W}} = \spant \parG{\widehat{N}_{ij}^{p_1, \,
		p_2}, \, 1 \leq i \leq n_1, \, 1 \leq j \leq n_2}~. 
\end{equation*}
It is now possible to introduce the following \ac{B-spline} vector
spaces defined on the reference domain $\igasrfpar$: 
\begin{equation}
\label{eq:cart_discr_sps_ref}
\begin{alignedat}{2}
&\CaSpBP_h^0 &&= S^{p_1, \, p_2}_{\alphav_1, \,
	\alphav_2}\parT{\Qch_h},\\ 
&\CaSpBP_h^1 &&= S^{p_1-1, \, p_2}_{\alphav_1-1, \,
	\alphav_2}\parT{\Qch_h} \times S^{p_1, \, p_2-1}_{\alphav_1, \,
	\alphav_2-1}\parT{\Qch_h},\\ 
&\CaSpBP_h^{1*} &&= S^{p_1, \, p_2-1}_{\alphav_1, \,
	\alphav_2-1}\parT{\Qch_h} \times S^{p_1-1, \, p_2}_{\alphav_1-1, \,
	\alphav_2}\parT{\Qch_h},\\ 
&\CaSpBP_h^2 &&= S^{p_1-1, \, p_2-1}_{\alphav_1-1, \,
	\alphav_2-1}\parT{\Qch_h}. 
\end{alignedat}
\end{equation}
These are the discrete conforming counterparts of the spaces
\begin{equation*}
\label{eq:cart_sps_ref}
\begin{alignedat}{3}
&\CaSpBP^{0} && = \Hsp(\widehat{\gradt};\, \igasrfpar)&& =
\parG{u \in \Lsp^2(\igasrfpar): \widehat{\gradt} \,u \in
	\Lsp^2(\igasrfpar; \, \R^2)},\\ 
&\CaSpBP^{1} && = \Hsp(\widehat{\curltsc};\, \igasrfpar) &&=
\parG{\uv \in \Lsp^2(\igasrfpar; \, \R^2): \widehat{\curltsc} \,\uv
	\in \Lsp^2(\igasrfpar)},\\ 
&\CaSpBP^{1*} && = \Hsp(\widehat{\divergt};\, \igasrfpar)&& =
\parG{\uv \in \Lsp^2(\igasrfpar; \, \R^2): \widehat{\divergt} \,\uv
	\in \Lsp^2(\igasrfpar)}, \\ 
& \CaSpBP^{2} && = \Lsp^2(\igasrfpar).&&
\end{alignedat}
\end{equation*}
Following \cite{buffa:2013}, it is then possible to define a set of projectors $\widehat{\Pi}^{k}$,
$k=0, 1, 1*, 2$ such that the following
diagrams are commutative
\begin{equation*}
\begin{split}
\tikzsetnextfilename{de_rham/de_rham_2d_cart_curl_hat}
\begin{tikzpicture}[commutative diagrams/every diagram,
arr/.style={commutative diagrams/.cd,
	every arrow,
	every label}]
\def\minW{25mm}
\def\minH{15mm}
\def\hShift{10mm}
\def\vShiftA{20mm}
\def\vShiftB{4mm}
\def\auxA{1mm}
\def\auxB{4mm}
%
\node (X0) [] at (0, 0) {$\CaSpBP^0$};
\node (X1) [xshift=1.1*\minW] at (X0) {$\CaSpBP^1$};
\node (X2) [xshift=1.1*\minW] at (X1) {$\CaSpBP^2$};
%
\node (X0h) [yshift=-\minH] at (X0) {$\CaSpBP_h^{0}$};
\node (X1h) [yshift=-\minH] at (X1) {$\CaSpBP_h^{1}$};
\node (X2h) [yshift=-\minH] at (X2) {$\CaSpBP_h^{2}$};
%
\node (R) [xshift=-0.6*\minW] at (X0) {$\R$};
\node (Z) [xshift=0.6*\minW] at (X2) {$0$};
\node (Rh) [yshift=-\minH] at (R) {$\R$};
\node (Zh) [yshift=-\minH] at (Z) {$0$};
\path[draw, arr] (X0) -- (X1) node [midway] {$\gradth$};
\path[draw, arr] (X1) -- (X2) node [midway] {$\curltsch$};
\path[draw, arr] (R) -- (X0);
\path[draw, arr] (X2) -- (Z);
\path[draw, arr] (X0h) -- (X1h) node [midway] {$\gradth$};
\path[draw, arr] (X1h) -- (X2h) node [midway] {$\curltsch$};
\path[draw, arr] (Rh) -- (X0h);
\path[draw, arr] (X2h) -- (Zh);
\path[draw, arr] (X0) -- (X0h) node [midway, anchor=west, yshift=0.8mm] {$\widehat{\Pi}^0$};
\path[draw, arr] (X1) -- (X1h) node [midway, anchor=west, yshift=0.8mm] {$\widehat{\Pi}^1$};
\path[draw, arr] (X2) -- (X2h) node [midway, anchor=west, yshift=0.8mm] {$\widehat{\Pi}^2$};
\end{tikzpicture}
\end{split}
\end{equation*}
\begin{equation*}
\begin{split}
\tikzsetnextfilename{de_rham/de_rham_2d_cart_div_hat}
\begin{tikzpicture}[commutative diagrams/every diagram,
arr/.style={commutative diagrams/.cd,
	every arrow,
	every label}]
\def\minW{25mm}
\def\minH{15mm}
\def\hShift{10mm}
\def\vShiftA{20mm}
\def\vShiftB{4mm}
\def\auxA{1mm}
\def\auxB{4mm}
%
\node (X0) [] at (0, 0) {$\CaSpBP^0$};
\node (X1) [xshift=1.1*\minW] at (X0) {$\CaSpBP^{1*}$};
\node (X2) [xshift=1.1*\minW] at (X1) {$\CaSpBP^2$};
%
\node (X0h) [yshift=-\minH] at (X0) {$\CaSpBP_h^{0}$};
\node (X1h) [yshift=-\minH] at (X1) {$\CaSpBP_h^{1*}$};
\node (X2h) [yshift=-\minH] at (X2) {$\CaSpBP_h^{2}$};
%
\node (R) [xshift=-0.6*\minW] at (X0) {$\R$};
\node (Z) [xshift=0.6*\minW] at (X2) {$0$};
\node (Rh) [yshift=-\minH] at (R) {$\R$};
\node (Zh) [yshift=-\minH] at (Z) {$0$};
\path[draw, arr] (X0) -- (X1) node [midway] {$\rotth$};
\path[draw, arr] (X1) -- (X2) node [midway] {$\divergth$};
\path[draw, arr] (R) -- (X0);
\path[draw, arr] (X2) -- (Z);
\path[draw, arr] (X0h) -- (X1h) node [midway] {$\rotth$};
\path[draw, arr] (X1h) -- (X2h) node [midway] {$\divergth$};
\path[draw, arr] (Rh) -- (X0h);
\path[draw, arr] (X2h) -- (Zh);
\path[draw, arr] (X0) -- (X0h) node [midway, anchor=west, yshift=0.8mm] {$\widehat{\Pi}^0$};
\path[draw, arr] (X1) -- (X1h) node [midway, anchor=west, yshift=0.8mm] {$\widehat{\Pi}^{1*}$};
\path[draw, arr] (X2) -- (X2h) node [midway, anchor=west, yshift=0.8mm] {$\widehat{\Pi}^2$};
\end{tikzpicture}
\end{split}
\end{equation*}
Once that the spaces are defined on the parametric domain, we turn our
attention to the physical domain $\srfdomflat$, which we assume to be
represented as $\srfdomflat = \vec{F}(\igasrfpar)$, where the
parametrization $\Fv$ is a \ac{NURBS} surface generated by a given
set of control points, which are the coefficients multiplying the
basis functions, see Figure~\ref{fig:map}.
\begin{figure}
	\centering
	\includegraphics{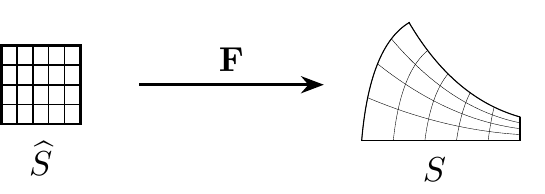}
	\caption{Geometry mapping from reference domain.}
	\label{fig:map}
\end{figure}
Given a mesh $\Qch_h$ on the reference domain, the parametrization
induces a mesh in the physical one: 
\begin{equation*}
\Qc_h = \parG{K \subset \srfdomflat:\, K = \Fv(Q), \, Q \in \Qch_h}.
\end{equation*}
We assume that the parametrization $\Fv$ is regular in the sense of
\cite[Assumption 3.1]{higueras:2016}, that is, we assume that it is a
bi-Lipschitz homeomorphism between $\igasrfpar$ and $\srfdomflat$,
$\eval{\Fv}_{\closure{Q}} \in C^\infty(\closure{Q})$, $\closure{Q} \in
\Qch_0$ and $\eval{\Fv^{-1}}_{\closure{{K}}} \in
C^\infty({\closure{K}})$, $\closure{{K}} \in \Qc_0$. 
With the introduced parametrization, the following pullbacks that
relate spaces on the physical domain to the corresponding ones in the
parametric domain are well-defined: 
\begin{equation}
\label{eq:pushback}
\begin{alignedat}{4}
&\iota^0&&: \CaSpB^0 \rightarrow \CaSpBP^0,  \quad&& v &&\mapsto v
\comp \Fv~, \\ 
&\iota^1&&: \CaSpB^1 \rightarrow \CaSpBP^1,  \quad&& \vv &&\mapsto
\jac{\Fv}^T\parT{\vv \comp \Fv}~,\\ 
&\iota^{1*}&&: \CaSpB^{1*} \rightarrow \CaSpBP^{1*},  \quad&& \vv
&&\mapsto \det\parT{\jac{\Fv}}\jac{\Fv}^{-1}\parT{\vv \comp \Fv}~, \\ 
&\iota^2&&: \CaSpB^2 \rightarrow \CaSpBP^2, \quad && v&& \mapsto
\det\parT{\jac{\Fv}}\parT{v \comp \Fv}~.\\ 
\end{alignedat}
\end{equation}
The discrete spaces on the physical domain are defined by
push-forward, \textit{i.e.} applying the inverse of~\eqref{eq:pushback} to the
discrete spaces on the reference domain~\eqref{eq:cart_discr_sps_ref}: 
\begin{equation*}
\label{cart_discr_sps}
\begin{alignedat}{2}
&\CaSpB_h^0 &&= \parG{v_h: \iota^0(v_h) \in \CaSpBP_h^0},\\
&\CaSpB_h^{1} &&= \parG{\vv_h: \iota^1(\vv_h) \in \CaSpBP_h^1},\\ 
&\CaSpB_h^{1*} &&= \parG{\vv_h: \iota^{1*}(\vv_h) \in
	\CaSpBP_h^{1*}},\\ 
&\CaSpB_h^2 &&= \parG{v_h: \iota^2(v_h) \in \CaSpBP_h^2}.
\end{alignedat}
\end{equation*}
Moreover, we assume that the regularity of the parametrization $\Fv$ is higher or equal to the one of the discrete spaces~\eqref{eq:cart_discr_sps_ref}~\cite{buffa:2011}.
Between the continuous and the discrete spaces it is possible to
define stable projectors such that the
diagrams~\eqref{eq:2D_curl_comp} and~\eqref{eq:2D_div_comp} are
commuting. Finally, the following estimates hold (2D analogous of
\cite[Corollary 5.12]{higueras:2016}): 
\begin{equation*}
\label{eq:iga_estimates}
\begin{aligned}
& \norm{u - \Pi^0 u}_{\Hsp^1(\srfdomflat)}  \leq C
h^{s}\norm{u}_{\Hsp^{s+1}(\srfdomflat)}, && u \in
\Hsp^{s+1}(\srfdomflat),\\ 
& \norm{\uv - \Pi^1 \uv}_{\Hsp(\curltsc;\,\srfdomflat)}\leq C
h^{s}\norm{\uv}_{\Hsp^{s}(\curltsc;\,S)}, && \uv \in
\Hsp^{s}(\curltsc;\,\srfdomflat),\\ 
&\norm{\uv - \Pi^{1*} \uv}_{\Hsp(\divergt;\,\srfdomflat)} \leq C
h^{s}\norm{\uv}_{\Hsp^{s}(\divergt;\,\srfdomflat)}, && \uv \in
\Hsp^{s}(\divergt;\,\srfdomflat),\\ 
&\norm{u - \Pi^2 u}_{\Lsp^2(\srfdomflat)} \leq C
h^{s}\norm{u}_{\Hsp^{s}(\srfdomflat)}, && u \in
\Hsp^{s}(\srfdomflat). 
\end{aligned}
\end{equation*}

\section{Numerical experiments}
\label{sec:tests}
We present the results of several numerical experiments carried out
with the discretization described in the previous sections. 
\subsection{Source problem}
\label{sec:source}
In this test case we consider the problem of reconstructing the magnetic flux $\Bv$ in a stationary case on an axisymmetric domain whose cross-section
is depicted in Figure~\ref{fig:pin_head_bc} (left).
\begin{figure}[!htb]
	\centering
	\includegraphics[width=0.34\textwidth]{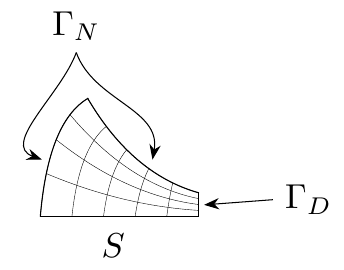} 
	\includegraphics[width=0.54\textwidth, trim=90mm 60mm 30mm
	80mm, clip]{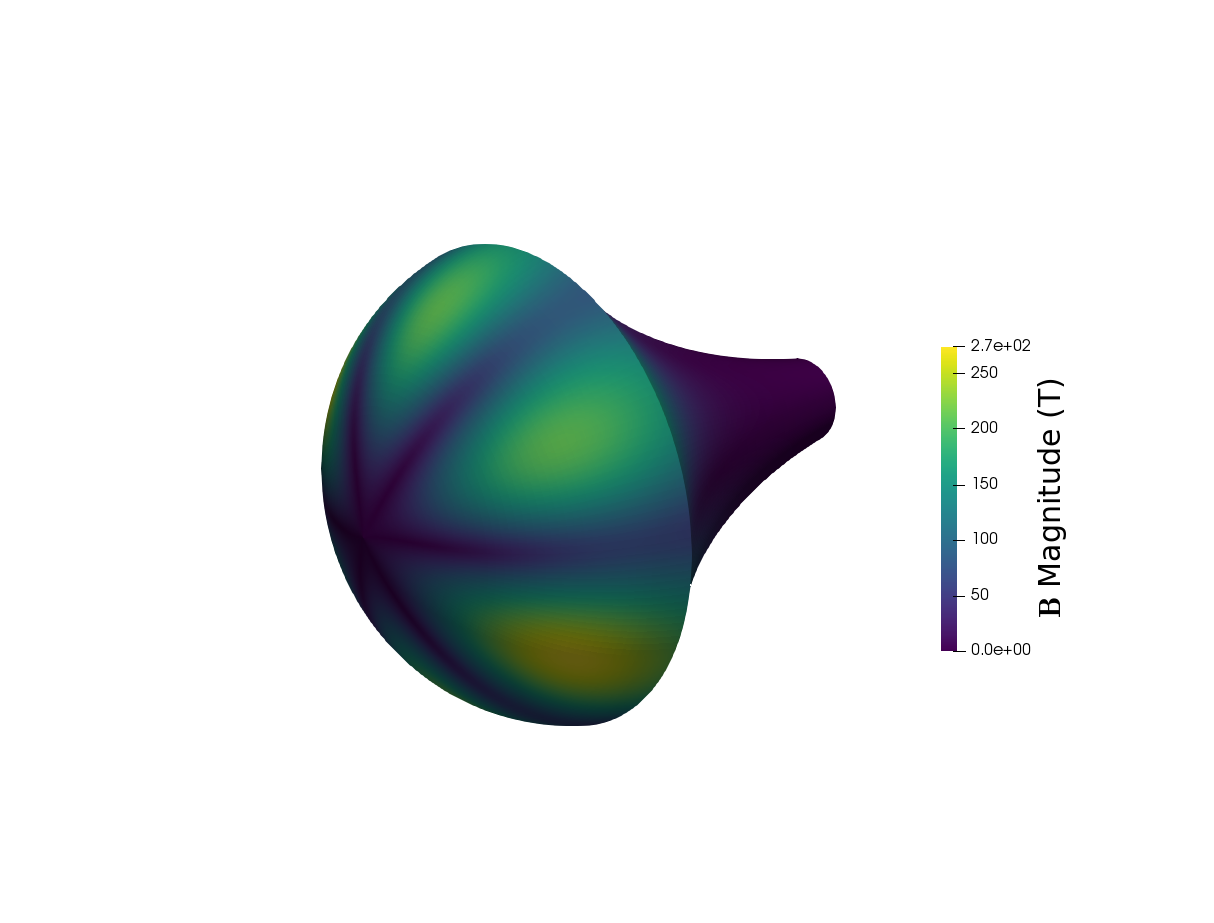} 
	\caption{Section of the computational domain with the
		boundaries associated to $\srfdom_N$ ($\Gamma_N$) and
		$\srfdom_D$ ($\Gamma_D$) (left) and magnitude of the
		magnetic flux density $\Bv$ on the domain (right).} 
	\label{fig:pin_head_bc}
\end{figure}%
To do so, we employ the so-called $\Av$-formulation, already introduced in \eqref{eq:A_form_st}, which amounts to compute the vector potential $\Av$ which can be then related to the flux by $\curlt \Av = \Bv$. 
Notice that the accurate computation of the vector potential is essential in modelling particle tracking in accelerator magnets, as
discussed in \cite{abele:2019}. 
Using a Coulomb gauge and imposing
homogeneous Dirichlet boundary conditions on $\srfdom_D$ regularized by the same Coulomb gauge 
and Neumann boundary conditions on $\srfdom_N$, we have that
\begin{equation*}
\label{pb:source_vp}
\begin{aligned}
\curlt \parT{\pmeab^{-1} \curlt \magvp}  &= \jv~,  &&\qquad \text{in}~
\voldomaxis~, \\ 
\divergt \parT{\pmitt \magvp} &= 0~, &&\qquad \text{in}~ \voldomaxis~,
\\ 
\parT{\pmeab^{-1}\curlt \magvp } \times \normal & = \jv_\srfdom~, &&\qquad \text{on}~ \srfdom_N~, \\ 
\magvp \times \normal &= \vec{0}~, &&\qquad \text{on}~ \srfdom_D~, \\
\parT{\pmitt \magvp} \cdot \normal & = 0~, &&\qquad \text{on}~ \srfdom_N~, \\ 
\divergt \parT{\pmitt \magvp} &= 0~, &&\qquad\text{on}~ \srfdom_D~,
\end{aligned}
\end{equation*}
where $\jv$ and $\jv_\srfdom$ are given source and surface current densities, respectively, while $\pmitt$ and $\pmeab$ are scalar constants, given by the corresponding values of these quantities in vacuum:
\begin{equation}
\label{eq:vacuum_epsmu}
\pmitt \approx \SI{8.8542e-12}{F.m^{-1}} ~, \qquad \pmeab =
\SI{4\pi e-7}{H.m^{-1}}~. 
\end{equation}
Note that, due to the fact that the permeability is constant, the magnetic
flux density and the magnetic field strength are related by the simple
relation $\Bv = \pmeab \Hv$. 
The method is tested using a manufactured solution which depends on a
parameter $\gamma$ and satisfies the homogeneous Dirichlet conditions
on $\srfdom_D$: 
\begin{equation*}
\magvp^{ex} = \begin{bmatrix}
\cos{3\theta} \, (5 - z)^3 \, \rho^{\gamma+1} \, e^{-\rho} \\
\rho^2 \,(\sin{\theta} - 2\sin{\theta}^3) \, (5-z)^\gamma\\
\sin{2\theta} \, (1-\cos(5 - z)) \, \rho^{\gamma+1}
\end{bmatrix}~.
\end{equation*}
The vector potential $\magvp^{ex}$ is not Coulomb gauged,
\textit{i.e.} $\divergt(\pmitt \magvp^{ex}) \neq 0$  and  will be different
from the solution of~\eqref{pb:source_vp}, but the magnetic
induction $\Bv = \curlt \magvp$, which is represented in
Figure~\ref{fig:pin_head_bc} (right), is gauge independent and is
used to compute the error.
A mixed formulation has been used to impose the Coulomb gauge and solve \eqref{pb:source_vp}. Using the method proposed in this work allows to solve a sequence of decoupled two-dimensional problems (one for each Fourier mode $m$) instead of a full three-dimensional one, with a significant advantage in terms of computational cost. In particular, consider the two finite-dimensional spaces, characterized by their basis functions,
\begin{align*}
\spant\parT{\parG{b_i}_{i=1}^{N_0}} & = \CySpT^{m, \, 0}_{h, \, \GammaA} \subset \CySpT_{\GammaA}^{m, \, 0} = \Hsp_\rho\parT{\gradtm}~,\\
\spant\parT{\parG{\vec{c}_i}_{i=1}^{N_1}} & = \CySpT^{m, \, 1}_{h, \, \GammaA}  \subset \CySpT_{\GammaA}^{m, \, 1} = \Hsp_\rho\parT{\curltm}~,
\end{align*}
where $N_0$ and $N_1$ indicate the dimensions of the corresponding finite-dimensional spaces and the subscript $\GammaA$ denote the boundary on which essential boundary conditions are imposed. Notice that, even if it is not explicit in the notation, the basis functions are different for each value of $m$.
The discrete mixed formulation for problem~\eqref{pb:source_vp} amounts to solving for each value of $m\neq0$ the following linear system:
\begin{equation}
\label{eq:small_lin_sys}
\begin{bmatrix}
\mat{A}_m & \mat{B}_m\\
\mat{B}_m^T  & \mat{0}
\end{bmatrix}
\begin{bmatrix}
\veca{u}_m\\
\veca{p}_m
\end{bmatrix}
=
\begin{bmatrix}
\veca{f}_m\\
\veca{0}
\end{bmatrix}
\end{equation}
where $\veca{u}_m$ and $\veca{p}_m$ are the \acp{DoF} associated to the discrete solution, 
\begin{align*}
\parQ{\mat{A}_m}_{i, \, j} & = \int \parT{\curltm \vec{c}_j} \cdot \parT{\curltm \vec{c}_i}\,\rho \dd\rho\dd z~,\\
\parQ{\mat{B}_m}_{i, \, j} & = \int \parT{\gradtm b_j} \cdot \vec{c}_i \,\rho \dd\rho\dd z
\end{align*}
and $\veca{f}_m$ is the term associated to the current density and the boundary term.
If all the \acp{DoF} corresponding to each mode are collected in a single vector $\veca{x}, $ we end up with  a linear system of the form shown in Figure~\ref{fig:big_sys},
\begin{figure}[!htb]
	\centering
	\includegraphics[width=0.40\textwidth]{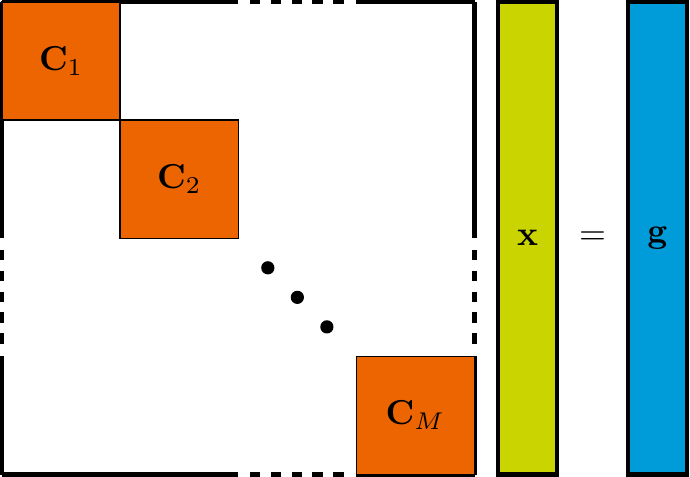}
	\caption{Block diagonal structure of the linear system.}
	\label{fig:big_sys}
\end{figure}
where each matrix $\mat{C}_m$ corresponds to that in~\eqref{eq:small_lin_sys} and $\veca{g}$ is obtained by the concatenation of all the right-hand side terms in~\eqref{eq:small_lin_sys}. For this testcase we choose the modes
\begin{equation*}
m = \pm 1, \,  \pm 2, \, \pm 3~,
\end{equation*}
so that $M=6$.
Note that the block diagonal structure in Figure~\ref{fig:big_sys} is obtained thanks to the Fourier basis and would not be obtained for a general choice of the basis in the angular direction.

In the left plot of Figure~\ref{fig:pin_head_Bl2_vs_sub}, we report
the error trend with respect to the number of subdivisions
($h$-refinement) for a smooth solution ($\gamma=2$). 
\begin{figure}[!htb]
	\centering
	\includegraphics[width=0.41\textwidth]{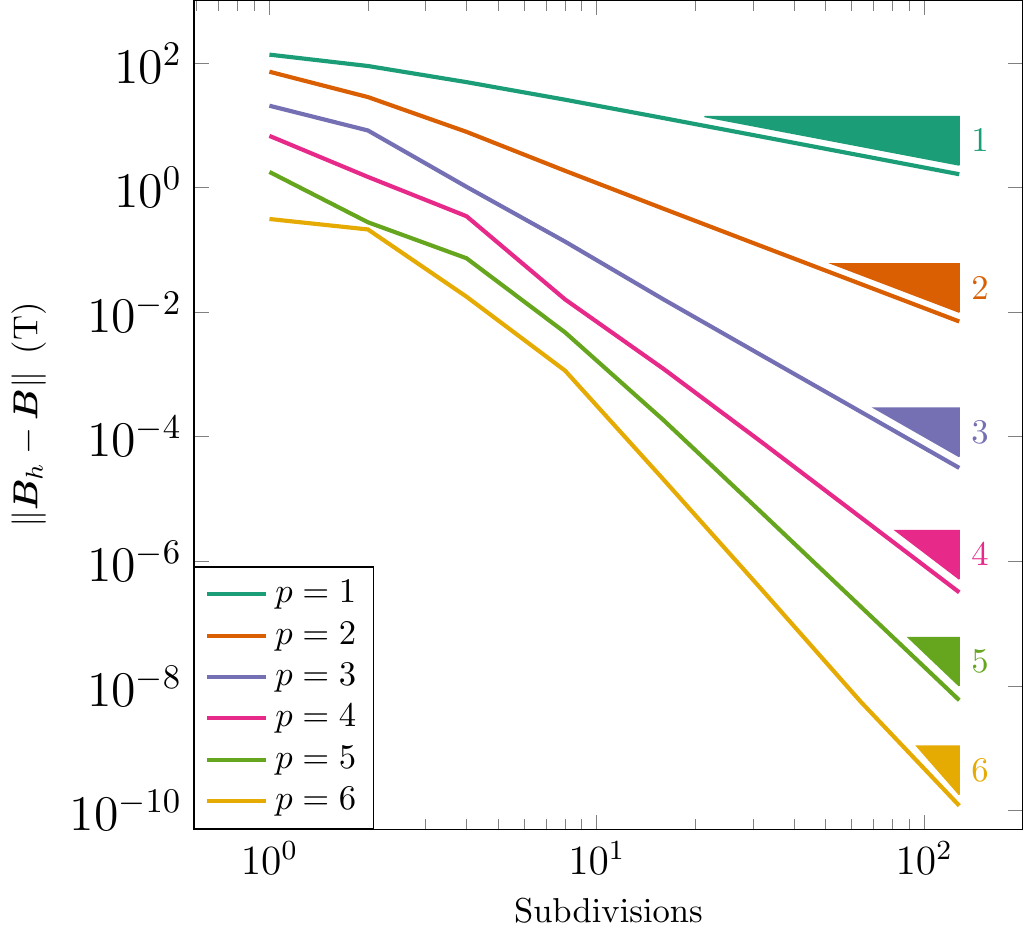}~~~~
	\includegraphics[width=0.41\textwidth]{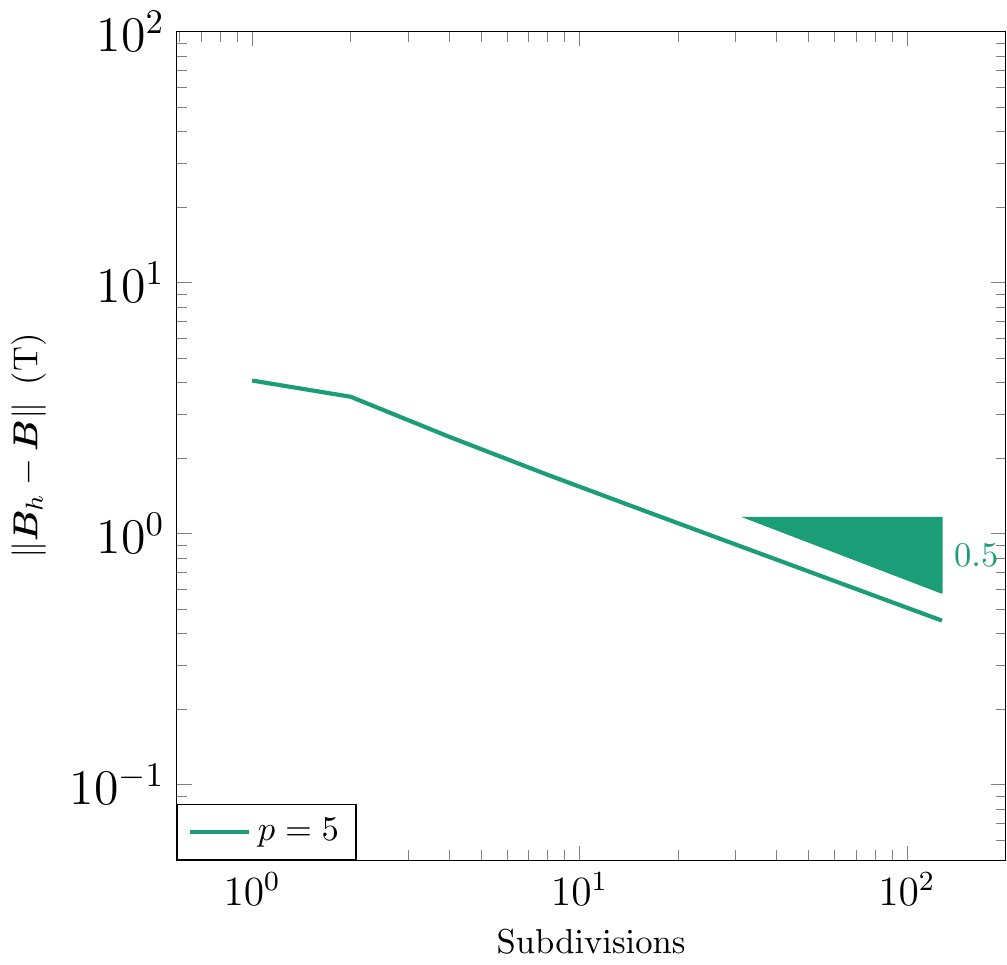} 
	\caption{Error on the magnetic field with respect to the
		number of subdivisions for a regular solution (left,
		$\gamma=2$) and for a non-smooth solution (right,
		$\gamma=0.5$) which limits the convergence rate.} 
	\label{fig:pin_head_Bl2_vs_sub}
\end{figure}
In this case the convergence rate is equal to the polynomial degree
$p$ of the basis functions. In the case of a less regular solution
with $\gamma=0.5$, we can see  instead in
Figure~\ref{fig:pin_head_Bl2_vs_sub} (right) that the low regularity
of the solution limits the convergence rate as expected. 

\subsection{Pillbox cavity}
\label{sec:pillbox}
In this section, we apply the method presented in this work
to the computation of the eigenvalues for a cylindrical cavity, known
in the literature as pillbox cavity \cite[\S8.7]{jackson:2007},  see Figure~\ref{fig:pillbox}. 
Given a pillbox cavity of radius $R=\SI{35}{mm}$ and length
$L=\SI{100}{mm}$, we want to compute
the eigenvalues and eigenfunctions for the electric  field strength $\Ev$,
satisfying the equations
\begin{equation}
\label{pb:eig_E}
\begin{alignedat}{2}
& \curlt \parT{\pmeab^{-1} \curlt\parT{\Ev}} = \omega^2 \pmitt \Ev~,
\qquad && \text{in $\voldomaxis$}~,\\ 
& \divergt \parT{\pmitt \Ev} = \vec{0}~,\qquad && \text{in
	$\voldomaxis$}~,\\ 
& \Ev \times \normal = \vec{0}~,\qquad && \text{in $\partial
	\voldomaxis$}~, 
\end{alignedat}
\end{equation}
where  $\omega^2$ is the eigenvalue associated to the eigenfunction
$\Ev$ that we will consider different from $0$, $\normal$ indicates the outer unit normal, $\pmitt$ and
$\pmeab$ are two scalar constants representing, respectively, the
permittivity and the permeability in vacuum  already defined in~\eqref{eq:vacuum_epsmu}. 
Since we consider $\omega^2\neq0$, the divergence constraint in~\eqref{pb:eig_E} is satisfied and can be neglected~\cite{boffi:2010}. 
The solutions of~\eqref{pb:eig_E} associated to eigenvalues different
from zero can be divided in two classes: one, {TM} (Transverse
Magnetic), is associated to fields where $B_z = 0$ everywhere and
$E_z=0$ on the lateral surface of the cylinder (see
Figure~\ref{fig:pillbox}, left). The second one, {TE} (Transverse
Electric), is associated to fields where $E_z = 0$ everywhere and
$\partial_\normal B_z = 0$ on the lateral surface of the cylinder see
Figure~\ref{fig:pillbox}, right) \cite[Chapter 8]{jackson:2007}. 
\begin{figure}[!htb]
	\centering
	\includegraphics[width=0.41\textwidth, trim=90mm 50mm 30mm 70mm,
	clip]{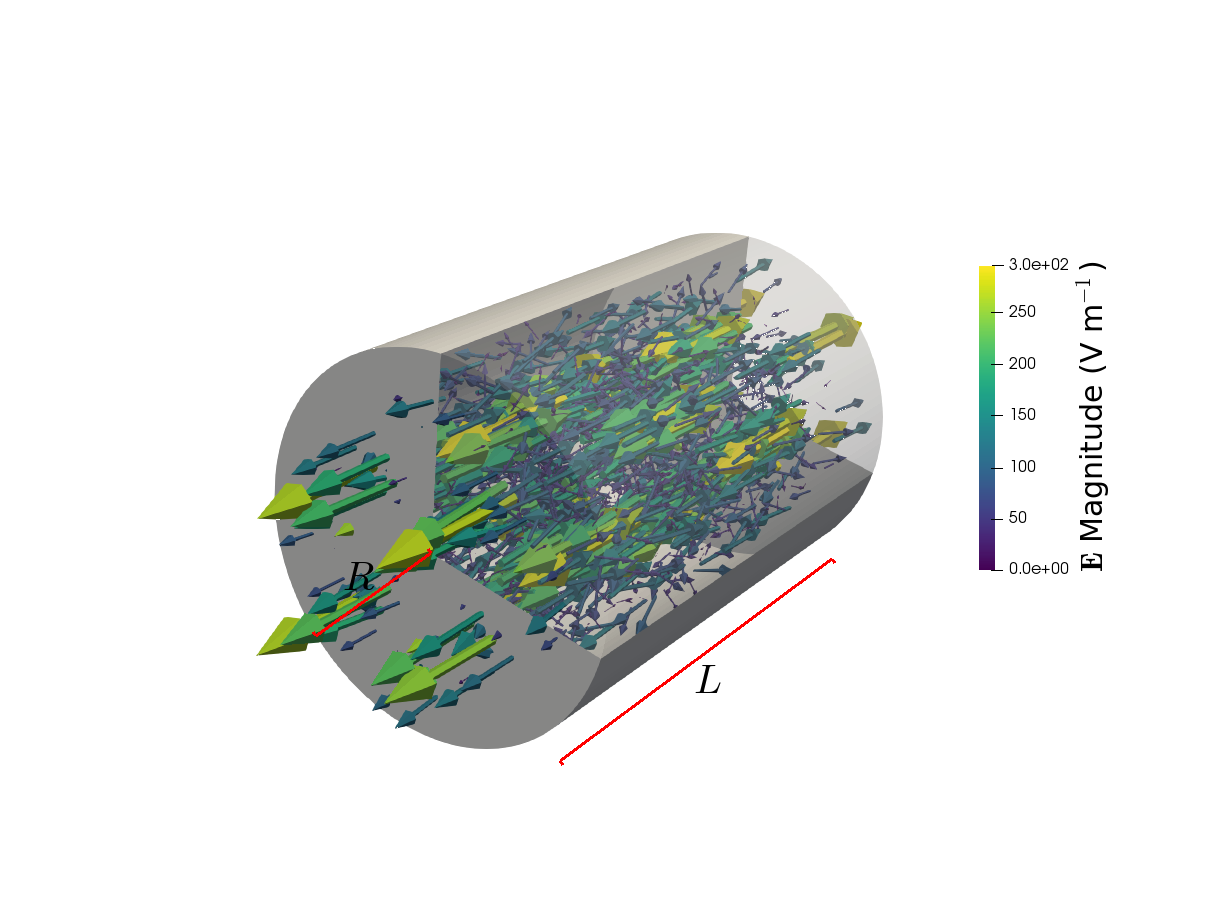}~~~ 
	\includegraphics[width=0.41\textwidth, trim=90mm 50mm 30mm 70mm,
	clip]{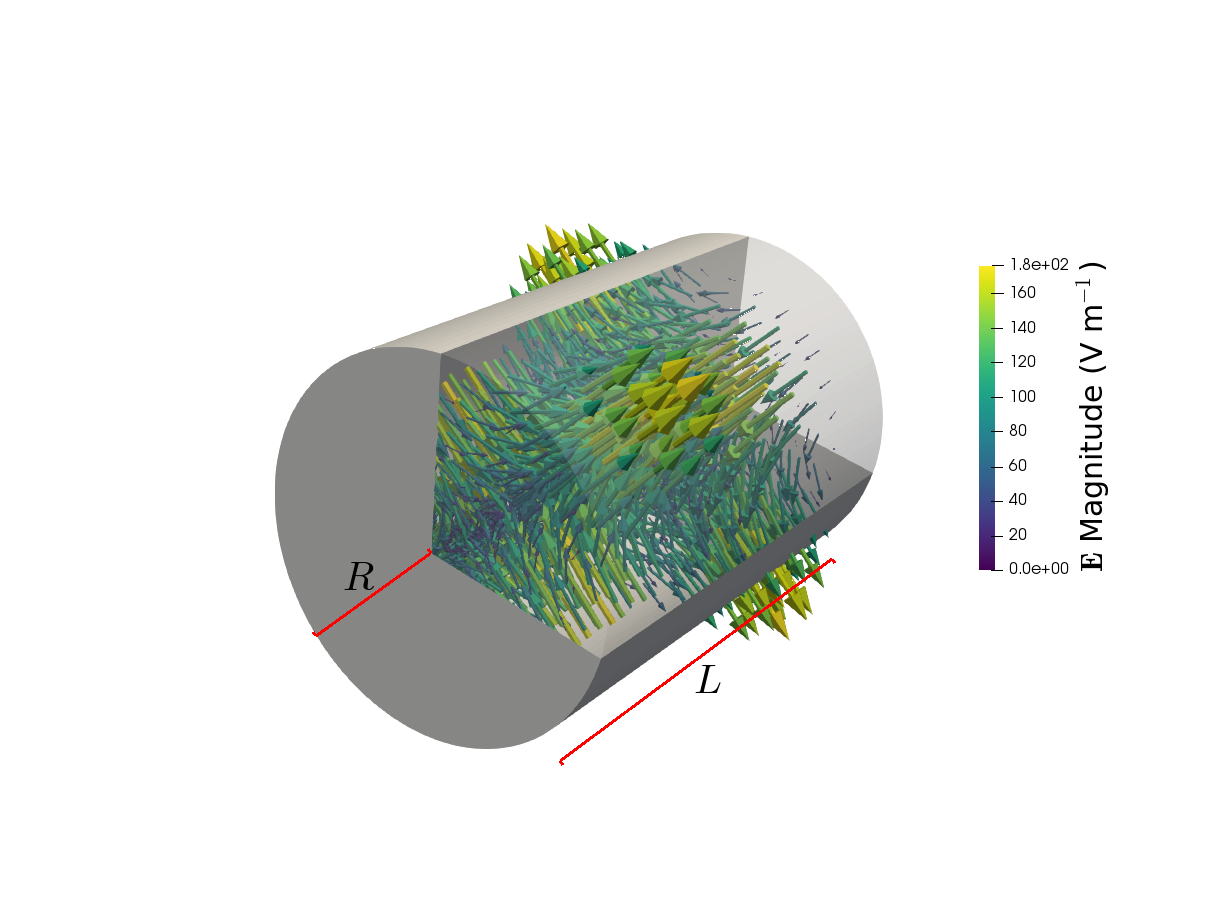} 
	\caption{Representation of a pillbox cavity of radius $R=\SI{35}{mm}$
		and length $L=\SI{100}{mm}$ with the electric field $\Ev$ associated
		to the eigenfunction $TM423$ (left) and $TE212$ (right).} 
	\label{fig:pillbox}
\end{figure}
The corresponding eigenvalues are given by
\begin{equation*}
\begin{aligned}
\omega^{TM}_{mnq} & = \dfrac{1}{\sqrt{\pmitt\pmeab}}
\sqrt{\dfrac{{\chi}_{mn}^2}{R^2} + \dfrac{q^2\pi^2}{L^2}}~, \qquad 
\omega^{TE}_{mnq} & = \dfrac{1}{\sqrt{\pmitt\pmeab}}
\sqrt{\dfrac{{\chi'}_{mn}^{2}}{R^2} + \dfrac{q^2\pi^2}{L^2}}~, 
\end{aligned}
\end{equation*}
where $\chi_{mn}$ and $\chi'_{mn}$ are the $n$-th root of the Bessel
function of the first kind of order $m$ and of its derivative,
respectively.  
The value $m$ is the Fourier mode of the corresponding
eigenfunction. As a consequence, for each $m$ we can compute the
eigenvalues varying the values of $q$ and $n$. In the following we
consider the case $\omega\neq0$. 

\medskip
Using cylindrical coordinates and applying the method presented in
this work, we end up with a set of independent two-dimensional
problems for each mode $m\neq0$ whose weak formulation is given by~\eqref{eq:axis_weak_eig_pb2}.
After the discretization step we have to find the
eigenpairs $(\lambda_{i, \, m}^2, \, \vva_{i, \, m})$
of linear system
\begin{equation*}
\mat{A}_m\vva_{i, \, m} = \lambda_{i, \, m}^2 \mat{M}_m \vva_{i, \, m}~,
\end{equation*}
where $\mat{A}$ is associated to the discretization on the curl-curl
operator, $\mat{M}$ is the mass matrix and $\lambda_{i, \, m}$ is an
approximation of $\omega_{mnq}$ for a specific value of $nq$.

In Figure~\ref{fig:pillbox_10eigs}, the exact value of the first $10$
angular frequencies for $m=26$ are represented by horizontal blue
lines. In the same figure, the approximated angular frequencies are
shown for different values of subdivisions of the parametric section
($h$-refinement) and using a uniform discretization and a maximum
degree $p=3$.  
\begin{figure}[!htb]
	\centering
	\includegraphics[width=0.45\textwidth]{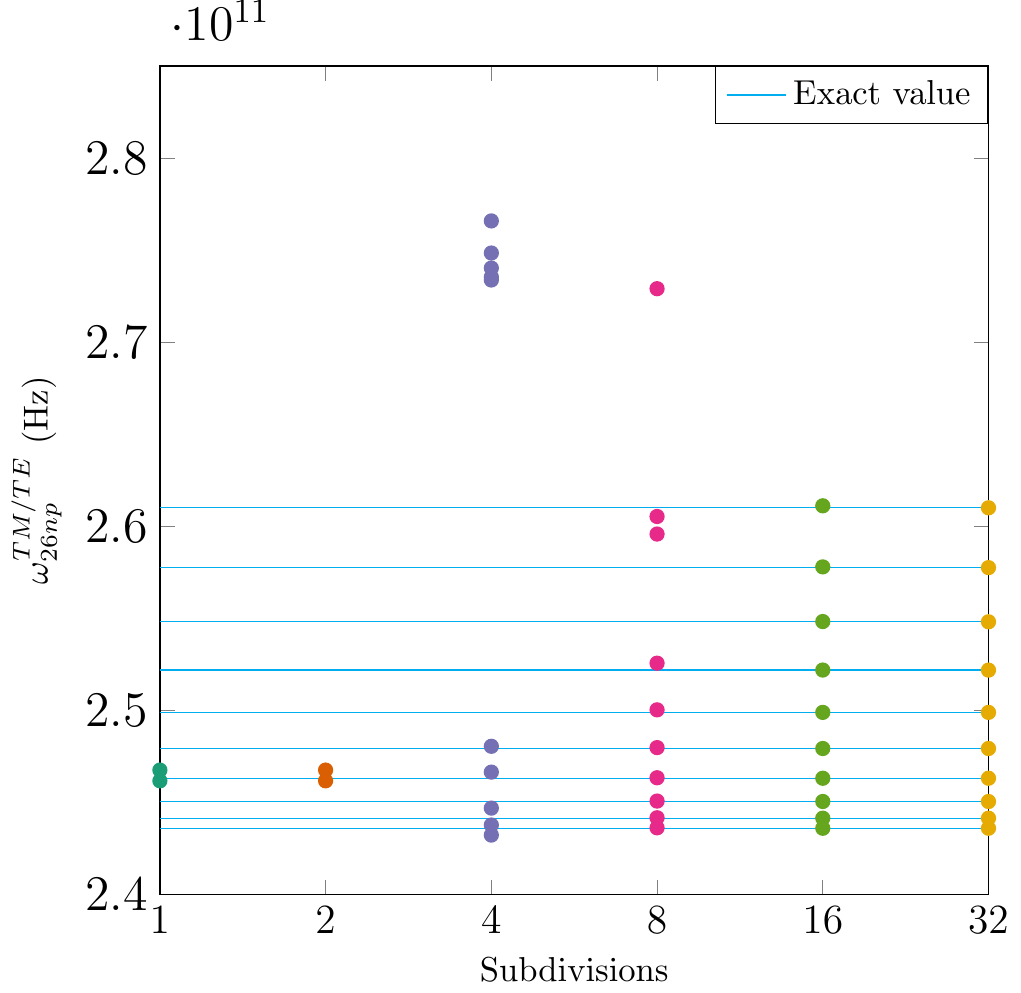}
	\caption{Approximations of the first $10$ angular frequencies
		($m=26$) for different number of subdivision of the
		parametric section. The exact values are represented by
		horizontal blue lines.} 
	\label{fig:pillbox_10eigs}
\end{figure}
It can be seen that no spurious modes appear and that the computed
eigenvalues converge to the exact ones. 
To test the approximation properties of our discretization, a specific
eigenvalue $\omega^{TE}_{134}$ has been chosen and the error for
different subdivisions and degrees $p$ has been computed. In
Figure~\ref{fig:pillbox_TE134}, the error trends with respect to the
to $h$-refinement (left) and to the number of \acp{DoF} (right) are
shown. 
\begin{figure}[!htb]
	\centering
	\includegraphics[width=0.41\textwidth]{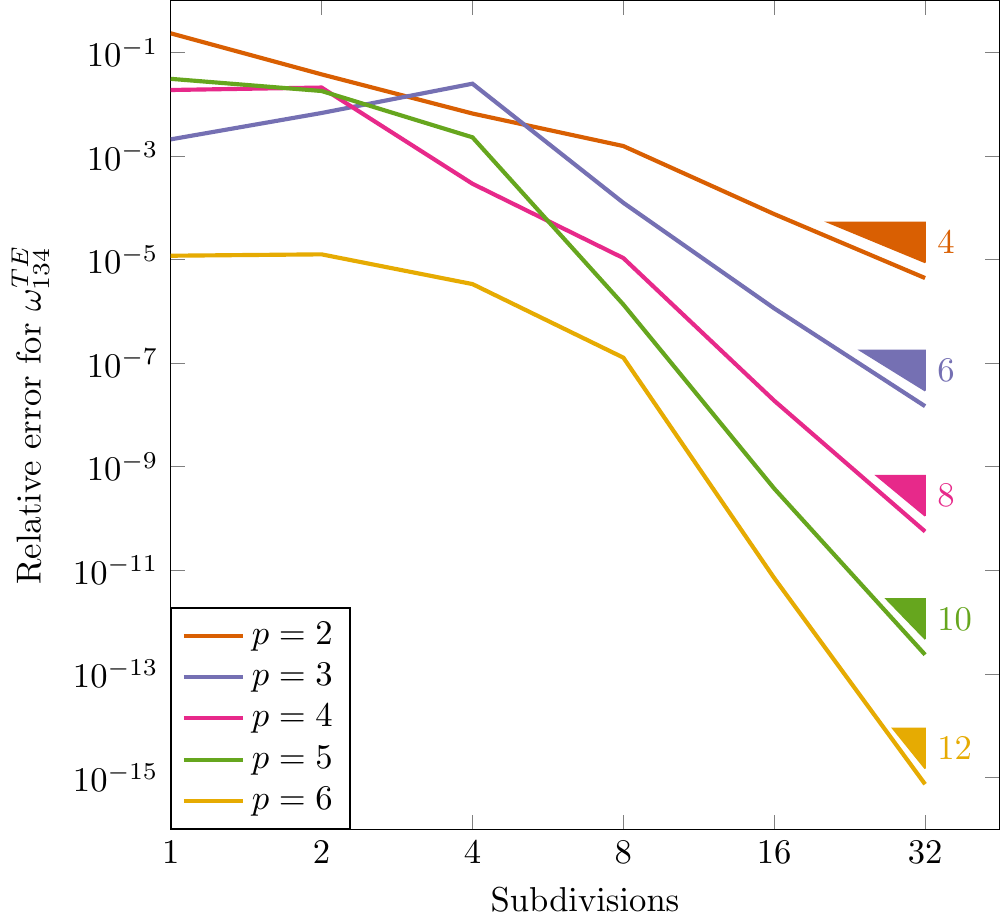} 
	\includegraphics[width=0.41\textwidth]{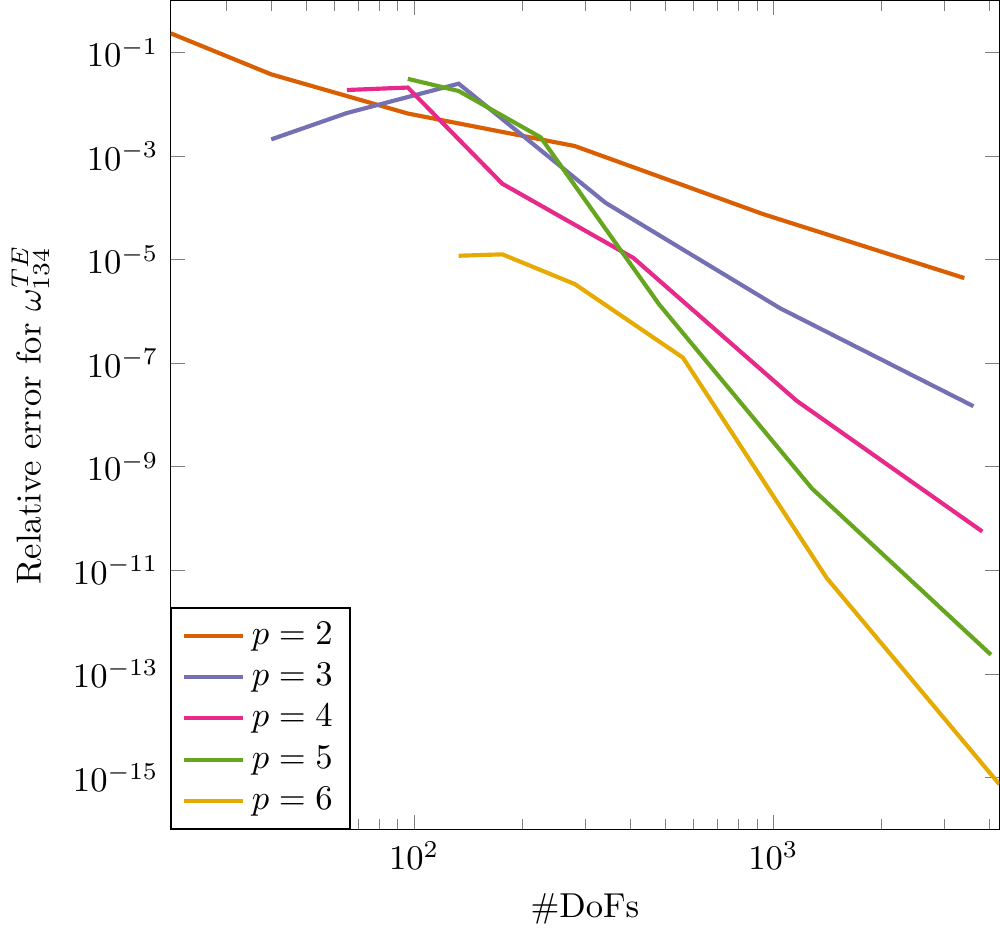}
	\caption{Relative error with respect to the number of
		subdivision of the parametric cross-section $\igasrfpar$
		(left) and with respect to the number of \acp{DoF} (right)
		for the angular frequency $\omega^{TE}_{134}$.} 
	\label{fig:pillbox_TE134}
\end{figure}
It can be seen that the approximate eigenvalues converges with a rate
equal to the double of the polynomial degree $p$ employed in the
discretization, which corresponds to the behaviour predicted by the theory \cite{arnold:2010}.

\subsection{TESLA cavity}
\label{sec:tesla_cavity}
In this section, we solve again the eigenvalue
problem~\eqref{pb:eig_E}, but on a domain which has a higher practical
relevance and whose geometry is defined using a \ac{NURBS}
surface. The specific design used in this test is the one-cell midcup
TESLA cavity whose precise definition can be found in~\cite[Table
III]{aune:2000}. 
We consider the problem of approximating the lowest resonant angular
frequency for the modes $m=1$ and $m=2$ which are, respectively, 
\begin{equation*}
\omega_1\approx
\SI{11468.32}{MHz}\quad\text{and}\quad \omega_2 =
\SI{14582.56}{MHz}~.  
\end{equation*}
In Figure~\ref{fig:tesla_eigfun}, the eigenfunctions associated to the
considered eigenvalues are depicted. 
\begin{figure}[!htb]
	\centering
	\includegraphics[width=0.41\textwidth, trim=10mm 30mm 30mm
	10mm, clip]{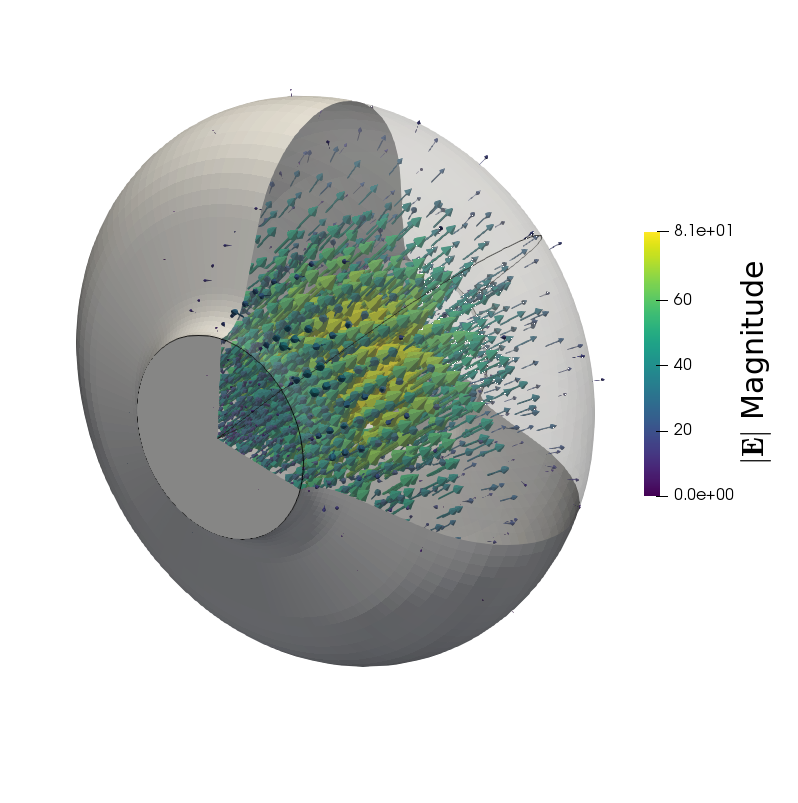}~~~ 
	\includegraphics[width=0.41\textwidth, trim=10mm 30mm 30mm
	10mm, clip]{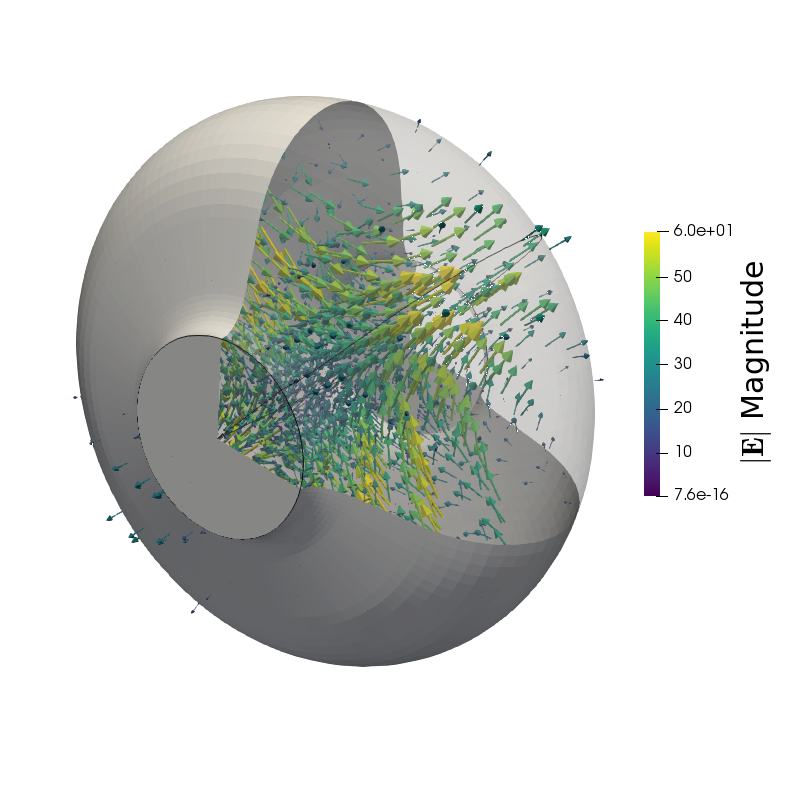} 
	\caption{TESLA cavity. Eigenfunctions associated to the lowest
		eigenvalues for the modes $m=1$ (left) and $m=2$ (right).} 
	\label{fig:tesla_eigfun}
\end{figure}%
In Figure~\ref{fig:tesla_cavity_err_vs_sub}, the trend of the relative
errors, estimated as the relative difference between two subsequent
refinement levels, are shown. 
\begin{figure}[!htb]
	\centering
	\includegraphics[width=0.41\textwidth]{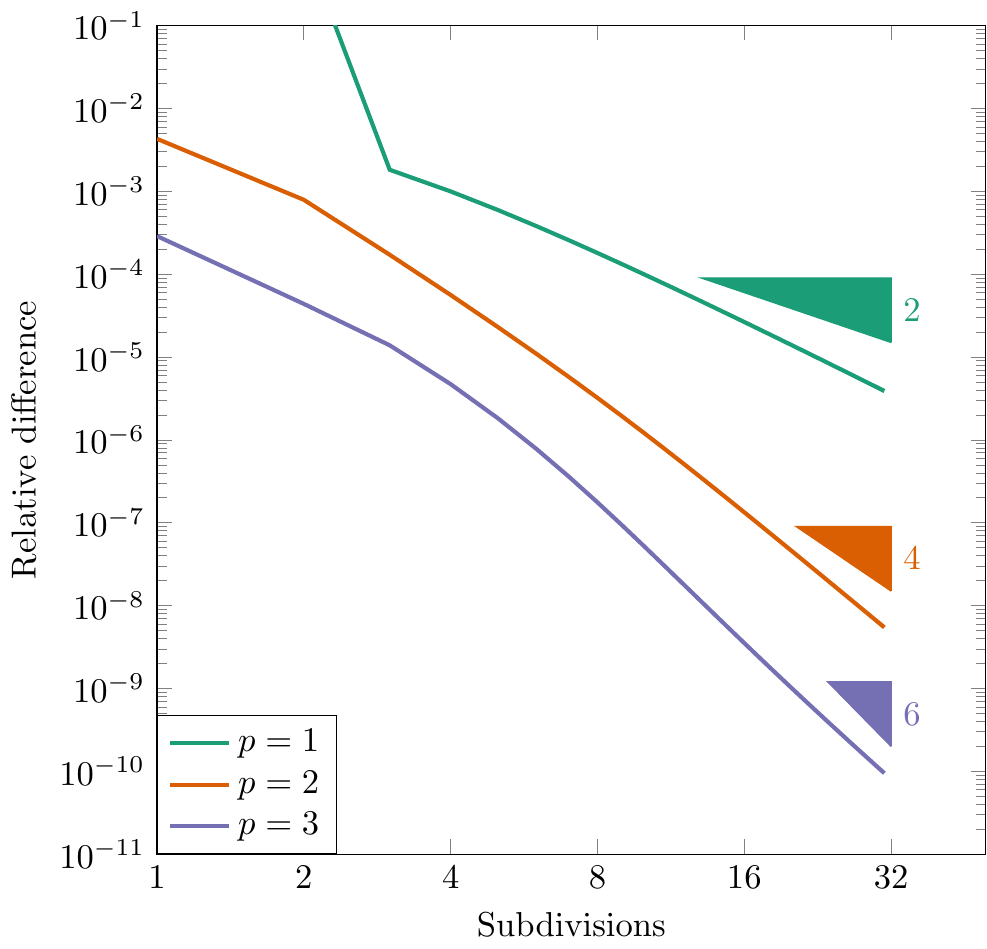} ~~~
	\includegraphics[width=0.41\textwidth]{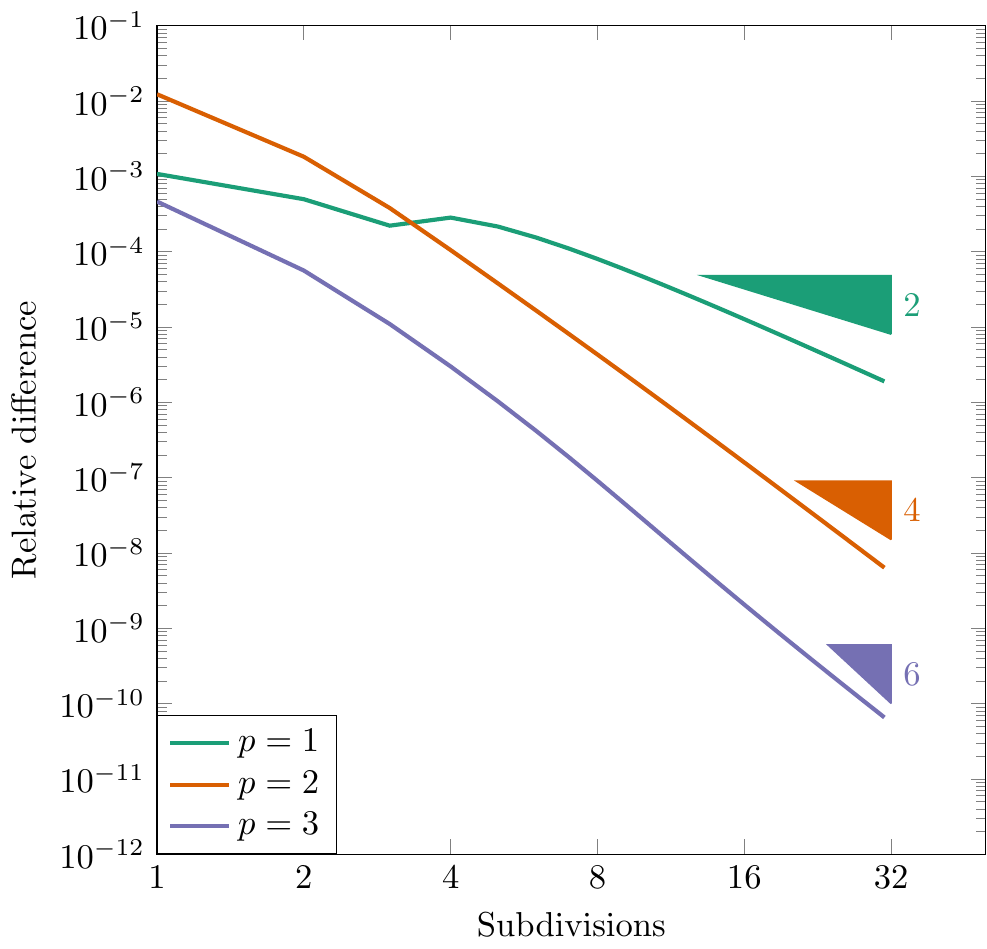}
	\caption{TESLA cavity. Error on the lowest eigenvalues $m=1$
		(left) and $m=2$ (right) with respect to the number of
		subdivisions for different degrees $p$ of the basis
		functions.} 
	\label{fig:tesla_cavity_err_vs_sub}
\end{figure}%
Also in this case, it can be seen that the approximate eigenvalues
converge with a rate close to twice the polynomial degree
$p$ employed in the discretization.

\section{Conclusions and future perspectives}
\label{sec:conclu}
\acresetall  
In the context of electromagnetic problems
on axisymmetric domains,
we have presented and analysed a  combination of  a spectral Fourier
approximation in the azimuthal direction with an 
\ac{IGA} approach in the radial and axial
directions. The resulting method provides
discrete approximations that, 
thanks to an appropriate choice of the finite-dimensional
approximation spaces,   preserves, at the discrete level, the structure 
of the continuous Maxwell equations. 
Rigorous error estimates have been obtained for the proposed method,
along with the proof that the associated discrete functional spaces
form a  de Rham  complex that is closed and exact.
A number of numerical benchmarks have been considered, yielding also empirical evidence that the resulting
method combines the capability of the \ac{IGA} approach to represent very accurately
complex geometries,  achieves   high convergence rates and 
allows to decouple the computation associated to different Fourier modes.

In future developments, it is planned to apply the proposed approach to the reconstruction of
magnetic fields to be  used for the simulation of particle accelerators, in order to complement
the results obtained in \cite{abele:2019} with high order particle tracking methods
with an equally accurate and physically sound representation of the  magnetostatic fields involved in these simulations.

\section*{Acknowledgements} 
This paper contains some of the results obtained by the first author
in  his PhD thesis 
work, which has been supervised jointly by the other authors. This
work has been carried out in the framework of 
a joint PhD agreement between Politecnico di Milano and Technische
Universit\"at Darmstadt. 
It has been partially supported by the ’Excellence Initiative’ of the German Federal and State Governments and the Graduate School of
Computational Engineering at Technische Universität Darmstadt.
We would like to thank Kersten Schmidt for the fruitful discussions
and Minah Oh for several explanations concerning the properties of
exact sequences of finite-dimensional spaces 
presented in her papers.



\bibliographystyle{elsarticle-num} 
\bibliography{axis_iga}





\end{document}